\documentclass[12pt,a4wide,oneside, reqno]{amsart}
\usepackage[T1]{fontenc}
\usepackage{amssymb,amsmath,color,textcomp,url,tikz,a4wide, 
mathtools}
\usepackage{etoolbox}

\usetikzlibrary{matrix,arrows}
\usetikzlibrary{fit}
\usepackage[labelformat=empty]{caption}

\usepackage{soul}

\usepackage{mdwlist}

\RequirePackage{ifpdf}
\ifpdf
\usepackage[pdftex]{hyperref}
\else
\usepackage[hypertex]{hyperref}
\fi

\usepackage{enumerate,xspace}

\def\showauthornotes{1}

\ifnum\showauthornotes=1
\newcommand{\Authornote}[2]{{\sf\small\color{blue}{[#1:
			#2]}}}
\else
\newcommand{\Authornote}[2]{}
\fi

\theoremstyle{plain}
\newtheorem{theorem}{Theorem}[section]
\newtheorem{cor}[theorem]{Corollary}
\newtheorem{prop}[theorem]{Proposition}
\newtheorem{lemma}[theorem]{Lemma}
\newtheorem*{teoA}{Theorem A}
\newtheorem*{teoB}{Theorem B}

\theoremstyle{definition}
\newtheorem{remark}[theorem]{Remark}
\newtheorem{fact}[theorem]{Fact}
\newtheorem{definition}[theorem]{Definition}
\newtheorem{example}[theorem]{Example}
\newtheorem*{question}{Question}

\newtheorem*{claim*}{Claim}
\newtheorem{claim}{Claim}
\newcounter{proofcount}
\AtBeginEnvironment{proof}{\stepcounter{proofcount}} 
\makeatletter    
\@addtoreset{claim}{proofcount}
\makeatother                   

\newenvironment{claimproof}[1][Proof of Claim \theclaim.] 
{%
	\proof[#1]%
	
}
{%
	\endproof%
}

\newenvironment{claimproof*}[1][Proof of Claim.] 
{%
	\proof[#1]%
	
}
{%
	\endproof%
}

\newcommand{\nc}{\newcommand}

\nc{\Z}{\mathbb{Z}}
\nc{\Q}{\mathbb{Q}}
\nc{\N}{\mathbb{N}}
\nc{\F}{\mathbb{F}}
\nc{\E}{\mathbb{E}}
\nc{\C}{\mathbb{C}}

\nc{\M}{\mathcal{M}}
\nc{\Fc}{\mathcal{F}}
\nc\LL{\mathcal L}
\nc{\Ult}{\mathfrak{U}}
\nc{\HH}{\mathcal{H}}
\nc\wh{\widehat}

\nc{\GP}{\operatorname{GP}}
\nc{\EE}{\mathcal{E}}

\nc{\dcl}{\operatorname{dcl}}

\nc{\acl}{\operatorname{acl}}

\nc{\nf}[1]{_{\mid {#1}}}
\nc{\restr}[1]{\!\!\upharpoonright_{#1}}
\nc{\sbgp}[1]{\langle\xspace {#1}\xspace\rangle}

\nc{\GO}[1]{G_{#1}^{00}}
\nc\inv{ ^{-1}}

\nc{\tp}{\operatorname{tp}}
\nc\cb{\operatorname{Cb}}
\nc\U{\operatorname{U}}
\nc{\cf}{\text{cf. }}
\nc{\eg}{\text{e.g. }}

\def\Ind#1#2{#1\setbox0=\hbox{$#1x$}\kern\wd0\hbox to
	0pt{\hss$#1\mid$\hss} \lower.9\ht0\hbox to
	0pt{\hss$#1\smile$\hss}\kern\wd0}
\def\Notind#1#2{#1\setbox0=\hbox{$#1x$}\kern\wd0\hbox to
	0pt{\mathchardef\nn="0236\hss$#1\nn$\kern1.4\wd0\hss}\hbox
	 to
	0pt{\hss$#1\mid$\hss}\lower.9\ht0 \hbox to
	0pt{\hss$#1\smile$\hss}\kern\wd0}

\def\indip{\mathop{\ \ \hbox to 0pt{\hss$\mid^{\hbox to
				0pt{$\scriptstyle P$\hss}}$\hss}
		\lower4pt\hbox to 0pt{\hss$\smile$\hss}\ \ }}
\def\nindip{\mathop{\ \ \hbox to 0pt{\hss$\!\not{\mid}^{\hbox to
				0pt{$\scriptstyle\, P$\hss}}$\hss}
		\lower4pt\hbox to 0pt{\hss$\smile$\hss}\ \ }}

\begin{document}
	
	\title{Stability, corners, and other 2-dimensional shapes}
	\date{\today}
	
	\author{Amador Martin-Pizarro, Daniel Palacin and Julia Wolf}
	\address{Abteilung f\"ur Mathematische Logik, Mathematisches 
	Institut,
		Albert-Ludwig-Universit\"at Freiburg, Ernst-Zermelo-Stra\ss e 1, 
		D-79104
		Freiburg, Germany}
			\email{pizarro@math.uni-freiburg.de}
	\address{Departamento de \'Algebra, Geometr\'ia y Topolog\'ia, 	
	Facultad de Ciencias Matem\'aticas, 
		Universidad Complutense de Madrid, Plaza Ciencias 3, 28040, 
		Madrid, Spain}
			\email{dpalacin@ucm.es}
	\address{Department of Pure Mathematics and Mathematical 
	Statistics, Centre for
		Mathematical Sciences, Wilberforce Road, Cambridge CB3 0WB, 
		United Kingdom}
	\email{julia.wolf@dpmms.cam.ac.uk}
	\thanks{The first two authors conducted research partially 
	supported by PID2020-116773GB-I00 as well as by the program GeoMod 
	ANR-19-CE40-0022-01 (ANR-DFG). The second author conducted research 
	partially supported by the grants 2020-T1/TIC-20313 and PID2021-122752NB-I00.}
	
	\begin{abstract}
		We introduce a relaxation of stability, called almost sure stability, which 
		is insensitive to perturbations by subsets of Loeb 
		measure $0$ in a non-standard finite group. We show that almost sure stability 
		satisfies a stationarity principle in the sense of geometric stability theory for measure independent elements. We 
		apply this principle to deduce the existence of squares 
		in dense almost surely stable subsets of Cartesian products of non-standard finite groups, possibly 
		non-abelian. Our results imply qualitative asymptotic versions for Cartesian products of finite groups. In the final section, we establish the existence of $3\times 2$-grids (and thus of $L$-shapes) in dense almost surely stable $2$-dimensional subsets of finite abelian groups of odd order. 
	\end{abstract}
	
	\maketitle
	
	\section{Introduction}

	In the last few years, local stability, as developed by Hrushovski and Pillay \cite{HP94}, has 
	found several applications \cite{TW19, CPT20, gC21, 
	MPPW21} to questions on asymptotic 
	behaviour in additive combinatorics. Recall that a  subset $A$ of a 
	group 
	$G$ is 
	$k$-\emph{stable} if its Cayley graph 
	 \[ \mathrm{Cay}(G,A)=\{ (g,h) \in G\times G  : \   
	g^{-1} \cdot h \in A\}\] induces no half-graph of height $k$, that is,  
	there is no sequence $(a_1,b_1,\ldots,a_k,b_k)$ in $G^{2k}$ such that the 
	pair $a_i\inv\cdot b_j$ belongs to $A$ if and only if $i\le j$. 
	
	Although the notion of stability has played and continues to play a crucial role in 
	the development of model theory, from a combinatorial perspective it is 
	exceptionally sensitive: the slightest perturbation of the set (adding a single instance of a half-graph) has the 
	potential to destroy it. 
	Motivated by this, we introduce a relaxation of stability which takes such 
	perturbations into account. Given finite sets $X$ and $Y$, a relation $S\subseteq X\times Y$ is \emph{almost surely $k$-stable} if there are few (in the 
	sense of the counting measure associated to $X\times Y$) half-graphs of 
	height $k$ induced by $S$ (see Definition \ref{D:robust}).  Similar 
	attempts to quantify stability appeared in earlier work of
	 Terry and the third author \cite[Section 5]{TW21} and Coregliano and Malliaris \cite[Definition 3.5]{CM22}.\footnote{More recently, and subsequent to the present work, this notion also appeared as $\mu$-stability in work of Chernikov and Towsner \cite[Remark 2.18]{CT24}.}
	
	Our purpose for this work is twofold: on the one hand, we aim to develop the notion of almost sure stability from a model-theoretic point of view.\footnote{Almost sure stability is taken further in the recent preprints \cite{aB25, mG25}.} On the other hand, we hope to further cement the relevance of stability to classical questions in arithmetic 
	combinatorics (in both the abelian and the non-commutative case). The paper thus focuses on the existence of certain 
	two-dimensional shapes, such as corners and  squares, in dense subsets of groups. 
	
    Given an abelian group $G$, a (non-trivial) \emph{corner} in a subset 
    $S\subseteq G\times G$ is given by a pair $(x, y)$ in $G^2$ along with $d\ne 0_G$ in $G$  with the property that all three pairs $(x,y)$, $(x+d, y)$, and $(x,y+d)$ 
lie in $S$. It is well known \cite{S13} that the existence of 
non-trivial corners for dense subsets in $G\times G$ implies Roth's theorem 
\cite{kR53} on 
$3$-term arithmetic progressions in dense subsets $A$ of $G$. Indeed, a non-trivial corner in $\mathrm{Cay}(G,A)\subseteq G\times G$ determined by $(x, y)$ and $d$ 
 gives rise, by a simple projection, to the non-trivial $3$-term arithmetic progression $x-y-d, x-y, 
x-y+d$ in $A$. 

 The existence of corners in subsets $S$ of $G\times G$ of positive 
density  follows straight from the multidimensional 
Szemer\'edi theorem. The latter was first proved qualitatively in 
\cite{FK78}, and quantitative versions (albeit with poor bounds) 
follow from the hypergraph regularity lemmas of Gowers 
\cite{G07} and Nagle, R\"odl, Schacht and Skokan 
\cite{NRS06,RS04}. However, finding corners is not quite as difficult as 
resolving the full multidimensional Szemer\'edi theorem, and indeed 
some remarkable progress has been made towards a strongly quantitative resolution \cite{ASz74,S05,S06,G05,LM07}.\footnote{See also the spectacular very recent preprint \cite{J25}.}
 
In the case of a non-abelian group $G$, we distinguish between 
so-called 
\emph{naive corners}, that is, configurations of the form 
$(x, y), (x \cdot g, y), (x, y \cdot g)$ with $g\neq 1_G$, and \emph{BMZ 
corners} (for Bergelson, McCutcheon and Zhang \cite{BMZ97}), which are of 
the  form $(x, y), (x \cdot g, y), (x \cdot g, y\cdot g)$ with $g\neq 1_G$. Note that the latter configuration can also be written as $(x, y), (g \cdot x, y), (x, y\cdot g)$, by first writing $z=x\cdot g$ and then replacing the set $S\subseteq G\times G$ with its image after applying the inverse operation in the first coordinate. Similarly, naive corners can be expressed as $(x,y)$, $(g\cdot x,y)$, $(g\cdot x, y\cdot g)$.
	
In abelian groups, naive and BMZ corners are easily seen to be equivalent. Some abelian techniques can be adapted to 
the non-abelian context for BMZ corners, but do not seem to apply to 
naive corners \cite{BMZ97}, and an extensive literature now exists \cite{S13,A16,BRZK17}.
	
Other 2-dimensional shapes are of interest in abelian groups, for instance \emph{squares}, which are configurations given by $(x, y)\in G^2$ together with $d\neq 0_G$ in $G$ such that all of $(x,y)$, $(x+d, y)$, $(x,y+d)$, $(x+d,y+d)$ lie in $S$. We observe that the non-abelian formulation of a square, where we require 
\[(x,y), (x\cdot g, y), (x, y\cdot g), (x\cdot g, y\cdot g)\]
to lie in $S$, actually contains both a naive and a BMZ corner (and could have been expressed in a number of equivalent ways).  
	
In the presence of almost sure stability, that is, when
few half-graphs of height $k$ are induced by the subset $S$ of $G\times G$, we are able to
obtain qualitative asymptotic bounds on the density of $S$ that guarantees the existence of a square in $S$. In fact, our methods work in both abelian and non-abelian finite groups (see Corollary \ref{C:corners}).

Moreover, we prove the following strengthening in the setting of finite groups of bounded exponent (see Corollary 
\ref{C:corners_bddexp}).

\begin{teoA}
	Given integers $k, r\ge 1$ and real numbers $\delta, \epsilon>0$ there exists 
	an integer 
	$N=N(k, r, \delta, \epsilon)\ge 1$ and real numbers $\theta=\theta(k, r,  
	\delta, \epsilon)>0$ and $\eta=(k, r,  
	\delta, \epsilon)>0$ with the following property. 
	
	Let $G$ be a finite group of 
	exponent at most $r$, and consider a relation $S\subseteq G\times G$ of size $|S|\ge \delta |G|^{2}$ such that the collection  
$\HH_k(S)$ of all 
half-graphs of height $k$ induced by $S$ on  $G$ has size $|\HH_k(S)|\le 
\theta  |G|^{2k}$. For an element $g$ in $G$,  set \[ \Lambda_\Box(S)_g=\{(a, 
	b)\in G^2 : (a,b), (a\cdot g, b), (a, b\cdot g), (a\cdot g,  
	b\cdot g) \in S \}.  \] 
Then there exists a subgroup $H$ of $G$ of index at most 
	$N$ such that
	 \[  
	 \left|\left\{ g\in H : \ |\Lambda_\Box(S)_g| <\eta |S|\right\} \right| <  \epsilon |H|.\]
\end{teoA}

Going beyond corners, in a recent breakthrough Peluse \cite{sP22} obtained the first reasonable bound for the existence of so-called 
\emph{$L$-shapes} (where each of the pairs $(x,y)$, $(x+d, y)$, $(x,y+d)$, $(x,y+2d)$ belongs to $S$) in dense subsets of $\mathbb{F}_p^n\times \mathbb{F}_p^n$.

In the presence of almost sure stability we are able to establish the existence of even larger $2$-dimensional patterns, including $3\times2$-grids (see Corollary \ref{C:Lshape}).

    \begin{teoB}
	Given an integer $k\ge 1$ and a real number $\delta>0$, there is an integer 
$N(k, \delta)\ge 1$ and real numbers $\theta=\theta(k, \delta)>0$ and  
$\epsilon=\epsilon(k, \delta)>0$ with the following property. 

Let $G$ be a finite abelian group of odd order with $|G|\ge N$, and consider a relation $S\subseteq G\times G$ of size $|S|\ge \delta |G|^{2}$ such that the collection  
$\HH_k(S)$ of all 
half-graphs of height $k$ induced by $S$ on  $G$ has size $|\HH_k(S)|\le \theta 
|G|^{2k}$. 

Then the set 
\[ \Lambda_{3\times 2}(S)= \left\{ (a, b, g) \in G^3 : 
\  
\parbox{8cm}{$(a, b), 
(a+ g, 
b), (a+2g, b),  (a, b+g), \\ (a+g,  b+ g) \text{ 
	and }  (a+ 2g, b+ g)  \text{ all lie in } S$}
\right\}\]
has size  $|\Lambda_{3\times 2}(S)| \ge \epsilon |G|^{3}$. In 
particular, the relation $S$ contains an $L$-shape. 
\end{teoB}

The techniques used to prove Theorems A and B are of 
   a model-theoretic flavor, and rely heavily on a stationarity result for 
    almost surely stable relations (see Theorem \ref{T:robuststable_main}).   In forthcoming  
	work, we establish some of the results presented here in the language 
	of additive combinatorics. While the techniques used will be different, the   
	structure of the proofs will follow a similar pattern. The present paper should therefore be viewed as an effort to highlight the close interactions and foster 
	stronger synergies between model theory and additive combinatorics. 
	
	In order to render the presentation of this first work more accessible to 
	an audience who may not be versed in the language of model theory, we 
	have adapted the results using some of the (perhaps more familiar) machinery of ultraproducts, 
	with the aim of keeping the presentation as self-contained as possible. This choice 
	should not present an obstacle for the model-theoretic reader, who will easily translate the terminology and techniques to the more general setting.
	
	\textbf{Outline of the paper.} Sections 2 and 3 introduce the background needed for the remainder of the paper. In Section 4 we define almost sure stability and prove a stationarity principle for almost stable relations. The main theorems concerning corners and squares are derived in Section 5, whilst Section 6 considers more complex 2-dimensional configurations.
	
	\textbf{Acknowledgement.} The authors are grateful for the anonymous referee for helpful comments and corrections.

	\section{Non-standard finite groups and Loeb 
	measures}\label{S:measure}

	Recall that a non-principal ultrafilter $\Ult$ on $\N$ is a 
	non-empty collection of infinite sets closed under finite 
	intersections and 
	with the property that either a subset of $\N$ or its complement 
	belongs to 
	$\Ult$.  Such ultrafilters exist and each one induces a finitely 
	additive probability measure 
	on all subsets of $\N$, taking values $0$ and $1$ only, such 
	that no 
	finite subset has measure $1$. 
	
	\begin{definition}\label{D:pseudofte}A \emph{non-standard finite 
	group} (sometimes referred to in the literature as a \emph{hyperfinite group}) has as underlying set $\prod_{n\to \Ult} G_n$ for  some 
		collection $(G_n)_{n\in \N}$ of finite groups of strictly 
		increasing size (that is $|G_n|<|G_{n+1}|$ for all $n$ in $\N$) 
		and some non-principal 
		ultrafilter  $\Ult$ on $\N$,  where the set 
		$\prod_{n\to \Ult} G_n$ consists of the infinite Cartesian 
		product  $\prod_{n\in \N} G_n$ modulo the equivalence relation 
		\[ (g_n)_{n\in\N} 
		\sim_\Ult (h_n)_{n\in\N}\ \Longleftrightarrow \{n\in \N  : \  
		g_n=h_n\} \in 
		\Ult.\] That is, we identify two sequences if they are equal 
		$\Ult$-almost everywhere (with respect to the measure 
		induced by $\Ult$ on all subsets of $\N$). 
	\end{definition}
	If we denote the class of $(g_n)_{n\in\N}$ modulo $\sim_\Ult$ by 
	$[(g_n)_{n\in\N}]_\Ult$, it follows immediately from the definition 
	of the 
	non-principal ultrafilter $\Ult$ that the quotient $\prod_{n\to 
	\Ult} G_n$ is a 
	group, where \[ 
	[(g_n)_{n\in\N}]_\Ult \cdot [(h_n)_{n\in\N}]_\Ult = [(g_n\cdot 
	h_n)_{n\in\N}]_\Ult,\] is well-defined and has neutral element 
	$1_{\prod_{\Ult} 
	G_n} = 
	[(1_{G_n})_{n\in\N}]_\Ult$. Furthermore, the inverse  
	$[(g_n)_{n\in\N}]_\Ult\inv= 
	[(g_n\inv)_{n\in\N}]_\Ult$. 
	
	Note that a finite Cartesian product of non-standard finite groups is 
	again 
	a non-standard finite group. 
	
	\begin{definition}\label{D:internal}
		A subset $X$ of a non-standard finite group $\prod_{n\to \Ult} 
		G_n$ is \emph{internal} if 
		there exists a collection $(X_n)_{n\in\N}$ such that each $X_n$ 
		is a subset of 
		$G_n$ and  \[X= \left\{ [(g_n)_{n\in\N}]_\Ult \in \prod_{n\to 
		\Ult} G_n :  \ g_n 
		\text{ belongs to } X_n \text{ for $\Ult$-almost all $n$  in } \N 
		\right\}.\]
	\end{definition}
	Given an internal subset $X$ of $G=\prod_{n\to \Ult} G_n$ 
	induced from the 
	collection $(X_n)_{n\in\N}$ as above, we will denote the subset 
	$X_n$ of $G_n$ 
	by $X(G_n)$. Likewise, we will sometimes denote the internal set 
	$X$ by $\prod_{n\to \Ult} X_n$.
	\begin{remark}\label{R:internal}~
		\begin{enumerate}[(a)]
			\item Every finite subset of a non-standard finite group is 
			internal.
			\item The collection of internal sets is closed under Boolean 
			combinations and 
			projections. Given a (group) word $w=w(u, v_1,\ldots, v_n)$ 
			and a tuple $\bar 
			a=(a_1,\ldots, a_n)$ in a non-standard finite group $G$, the set 
			\[X=\left\{ g\in G  : \  
			w(g, \bar a)=1_G\right\} \]   is internal.
			\item Given internal subsets $X_1,\ldots,X_n$ of a 
			non-standard finite group $G$, the set $X_1\times\dots\times 
			X_n$ is an internal subset of the $n$th Cartesian product 
			$G^n=G\times\dots\times G$. 
			\item Given an internal set $X$ of $G$, the subset
			\[
			\left\{ (x,y) \in G\times G  : \  x\cdot y\in X \right\}
			\]
			is an internal subset of $G\times G$. In particular, every translate (left or right) of an 
			internal set is internal. 
			\item In this particular set-up, \L o\'s's Theorem becomes 
			tautologically immediate: given an internal set $X$ in a 
			non-standard finite group $G= 
			\prod_{n\to\Ult} G_n$,  we have that \[ X\ne \emptyset \ 
			\Longleftrightarrow \  
			X(G_n) \ne \emptyset \text{ for $\Ult$-almost all $n$  in } \N. 
			\]
		\end{enumerate}
	\end{remark}
	
	Every non-standard finite group arising as a limit of finite groups 
 has 
cardinality continuum 
	and so it has continuum many internal sets. However, whenever we 
	restrict our attention to only countably many internal sets 
	non-standard finite groups have the following remarkable property, 
	which in model-theoretic terms is called {\em 
	$\aleph_1$-saturation}.
	
	\begin{fact}[$\aleph_1$-saturation]\label{F:Saturation}
		Let $G$ be a non-standard finite group. Every internal countable 
		cover of an internal set admits a finite sub-covering: given an 
		internal subset $X$ of $G$ and a countable family $(Y_n)_{n\in 
		\N}$ of internal subsets of $G$ such that $X=\bigcup_{n\in\N} 
		Y_n$, there is some natural number $k$ such that 
		\[
		X=Y_{n_1} \cup \dots \cup Y_{n_k}.
		\]
		Equivalently, a countable intersection of internal set is 
		non-empty whenever every finite sub-intersection is. 
	\end{fact}
	In particular, every infinite internal subset of a 
	non-standard 
	finite group must be uncountable.

	Every non-standard finite group is equipped with a finitely 
	additive probability measure on the Boolean algebra of internal 
	sets, induced by the normalised counting measure on every finite 
	group.
	\begin{definition}\label{D:Measure}
		The {\em Loeb} or {\em non-standard counting measure} of the 
		non-standard finite group $G=\prod_{n\to\Ult} G_n$ is defined 
		for every internal subset $X$ of $G$ as 
		\[
		\mu_G (X) = \lim_{n\to \Ult} \frac{|X(G_n)|}{|G_n|}.
		\]
		In an abuse of notation, we call the value $\mu_G(X)$ the {\em 
		density of $X$}.   
	\end{definition}
	Note that for every real number $r$ in the interval $[0,1]$,
	the measure $\mu_G(X)\ge r$ whenever  $|X(G_n)|\ge r|G_n|$ 
	for 
	$\Ult$-almost all $n$ in $\N$. 
	
	\begin{remark}\label{R:fibreDef}
		Given an internal subset $Z$ of  the non-standard finite group 
		$G^{k+m}$ and a real number $r$ in $[0,1]$, there is an internal 
		subset $Y_r$ of $G^m$ such 
		that 
		\[
		\left\{ y \in G^m  : \  \mu_{G^k} (Z_y) > r  \right\} \subseteq Y_r
		\subseteq \left\{ y \in G^m  : \  \mu_{G^k} (Z_y)\ge r  \right\},
		\]
		where $Z_y=\{x \in G^k  : \  (x,y)\in Z\}$ denotes the 
		fibre of $Z$ over $y$. Namely, set 
		\[
		Y_r = \prod_{n\to \Ult} \left\{ y\in G_n^m  : \  |Z_y(G_n^k)| \ge r 
		|G_n^k| \right\}.
		\]
	\end{remark}
	
	We include the following easy result to demonstrate how information about the measure of non-standard finite objects allows us to provide asymptotic bounds on the cardinalities of the corresponding sequence of finite objects.
	
	\begin{lemma}\label{L:Transfer}
		Consider a family $(G_n,X_n)_{n\in\N}$, where $X_n$ is a 
		subset of the finite group $G_n$. The following are equivalent for 
		every real number $r$ in $[0,1]$.
		\begin{enumerate}[(a)]
			\item For every $\epsilon>0$ there is some $n_0=n_0(r, \epsilon)$ 
			in $\N$  such that $|X_n| \le 
			(r+\epsilon) |G_n|$ for all $n > n_0$.
			\item For every non-principal ultrafilter $\Ult$ on $\N$, we 
			have that $\mu_G(X) \le r$ for the internal set  
			$X=\prod_{n\to\Ult} X_n$ in the non-standard finite group 
			$\prod_{n\to\Ult}G_n$.
		\end{enumerate}
	\end{lemma}
	\begin{proof}
		$(a)\Rightarrow(b):$ No finite set lies in a non-principal 
		ultrafilter $\Ult$. Hence, by the way the non-standard counting 
		measure on the non-standard finite group $G=\prod_{n\to\Ult} 
		G_n$ has been defined, it follows from (a) that $\mu_G(X)\le 
		r+\epsilon$ for every $\epsilon>0$, so $\mu_G(X)\le r$, as desired.
		
		\noindent $(b)\Rightarrow(a):$ If (a) is false, negating quantifiers, 
		there is some value $\epsilon>0$ such that for every natural number 
		$n_0\ge 1$, we find some $n > n_0$ with  $|X_n| > (r+ \epsilon)
		|G_n|$. Hence the subset
		\[
		Q_{\epsilon} = \left\{ n\in \N  : \   |X_n| > 
		(r+\epsilon) |G_n|  \right\}
		\]
		is infinite. Choose a non-principal ultrafilter $\Ult$ containing 
		each $Q_\epsilon$ (which exists). By 
		definition of the non-standard counting measure 
		on the non-standard finite group $G=\prod_{n\to\Ult} G_n$ we 
		deduce
		\[
		\mu_G(X) = \lim_{n\to\Ult} \frac{|X(G_n)|}{|G_n|} \ge 
		\lim_{n\to\Ult} r+\epsilon = r +\epsilon.  
		\]
	 Thus $\mu_G(X)> r$, 
		which yields the desired contradiction.
	\end{proof}
\begin{remark}\label{R:Transfer}
Taking set-theoretic complements, it follows from Lemma \ref{L:Transfer} that 
the following two conditions are equivalent.
\begin{enumerate}[(a)]
	\item For every $\epsilon>0$ there is some $n_0=n_0(r, \epsilon)$ 
in $\N$  such that $|X_n| \ge 
(r-\epsilon)|G_n|$ for all $n > n_0$.
\item For every non-principal ultrafilter $\Ult$ on $\N$, we 
have that $\mu_G(X) \ge r$ for the internal set  
$X=\prod_{n\to\Ult} X_n$ in the non-standard finite group 
$\prod_{n\to\Ult}G_n$.
\end{enumerate}
\end{remark}
	
	\begin{remark}\textup{(cf. }\cite[Section 
	2]{BT14}\textup{)}\label{R:Fubini}
		Carath\'eodory's  Extension Criterion is satisfied by the Loeb 
		measure by $\aleph_1$-saturation (Fact \ref{F:Saturation}). 
		Therefore, the Loeb measure $\mu_G$ extends to a 
		unique $\sigma$-additive measure on the $\sigma$-algebra 
		generated by the internal subsets 
		of  $G$. Abusing notation, we will not distinguish between 
		$\mu_G$ and its unique extension. 
		
		Furthermore, the family of measures $\{\mu_{G^k}\}_{k\ge 1}$ 
		satisfies the Fubini-Tonelli Theorem \cite[Theorem 19]{BT14}, meaning that for any
 internal subset $Z$ of $G^{n+m}$ the following holds.
		\begin{itemize}
			\item The function $y\mapsto \mu_{G^n}(Z_y)$, resp. 
			$x\mapsto \mu_{G^m}(Z_x)$,  is $\mu_{G^m}$-measurable, 
			resp. $\mu_{G^n}$-measurable.
			\item We have the equality 
			\[
			\mu_{G^{n+m}}(Z) = \int_{G^n} \mu_{G^m}(Z_x) \, 
			\mathrm{d}\mu_{G^n} =\int_{G^m} \mu_{G^n}(Z_y) \, 
			\mathrm{d}\mu_{G^m}.
			\]
		\end{itemize}
		
	\end{remark}

	\section{Tame families and dense subsets}\label{S:tame}

	In view of $\aleph_1$-saturation (Fact \ref{F:Saturation}), we 
	introduce the following notion, 
	which appears in \cite{BB21} as $\bigwedge$-internal and in 
	\cite{lvD15} as $\Pi$-definable. For readers familiar with model 
	theory this corresponds to type-definable sets over a  countable 
	parameter set in a countable language.
	
	\begin{definition}\label{D:InfInt}
		A subset $X$ of $G$ is {\em $\omega$-internal} if it is a countable 
		intersection of internal subsets of $G$. 
	\end{definition}
	
	\begin{remark}\label{R:InfInt}~
		\begin{enumerate}[(a)]
			\item Internal sets are $\omega$-internal. 
			\item Every countable decreasing chain $(Y_n)_{n\in \N}$ of 
			infinite internal subsets of a non-standard finite group $G$ 
			yields a non-empty $\omega$-internal set. By 
			$\aleph_1$-saturation (Fact \ref{F:Saturation}), this 
			$\omega$-internal set is uncountable.
   \item The projection of an $\omega$-internal subset of $G^{n+m}$ onto the first $n$ coordinates is $\omega$-internal in $G^n$. Indeed, if $X$ is a countable intersection of a decreasing family $(X_k)_{k\in\N}$ of internal subsets of $G^{n+1}$, Fact \ref{F:Saturation} gives immediately that $\pi(X)$ equals $\bigcap_{k\in\N} \pi(X_k)$.
 		\end{enumerate}
	\end{remark}
	
	\begin{example}
		Let $G$ be the non-standard finite group $\prod_{n\to \Ult} 
		\Z/2^n\Z$, for some non-principal ultrafilter $\Ult$. For every 
		natural number $k\ge 1$, the function $f_k:x\mapsto 2^k\cdot 
		x$ from $G$ to $G$ is an internal function, that is, its graph is an 
		internal set. Each internal subgroup $\mathrm{Im}(f_k)$ has 
		finite index in $G$ and they form a strictly decreasing chain, 
		which yields an $\omega$-internal uncountable subgroup of $G$ 
		which is $2$-divisible.   
	\end{example}
	
	In view of Fact \ref{F:Saturation} we will need to restrict our 
	attention to suitable countable families of internal sets. 
	Model-theoretically this is done by fixing a suitable countable 
	first-order language and considering the corresponding definable 
	sets, which are 
	always internal, see for example \cite{eH12,lvD15,dP20,BB21}. 
	For the sake of the presentation, we will avoid introducing the 
	terminology of first-order formulae and provide an approach 
	tailored to our purposes.
	
	\begin{definition}\label{D:Tame}
		Let $G=\prod_{n\to \Ult} G_n$ be a non-standard finite group. A 
		family $\mathcal F=\bigcup_{n\in \N} \mathcal F_n$, where each $\mathcal F_n$ consists of internal subsets of $G^n$,  is {\em tame} if it satisfies the following 
		conditions:
		\begin{enumerate}[(I)]
			\item Every internal set of the form 
			\[
			X_w=\left\{ (g_1,\ldots,g_n)\in G^n  : \  
			w(g_1,\ldots,g_n)=1_G\right\}, \]
			where $w=w(u_1,\ldots,u_n)$ is a (group) word, belongs to 
			$\mathcal F_n$.
			\item $\mathcal F$ is closed under finite Cartesian products, Boolean 
			combinations, projections and permutations of the coordinates.
			\item Given an internal subset $Z$ of $G^n$ in $\mathcal F_n$, 
			the internal set
			\[
			\left\{ (x,y)\in G^n\times G^n  : \  x\cdot y\in Z \right\}
			\]  
			belongs to $\mathcal F_{2n}$.
			\item Given an internal subset $Z$ of $G^{k+m}$ in $\mathcal 
			F_{k+m}$ and a rational number $q$ in $[0,1]$, there is an internal 
			set $Y_q$ in $\mathcal F_m$  such that 
			\[
			\left\{ y \in G^m  : \  \mu_{G^k} (Z_y) > q  \right\} \subseteq 
			Y _q \subseteq \left\{ y \in G^m  : \  \mu_{G^k} (Z_y)\ge q  
			\right\},
			\]
			as in Remark \ref{R:fibreDef}. In particular, for every internal 
			subset $Z$ of $G^{k+m}$ in $\mathcal F_{k+m}$, the subset  \[  
			\left\{ y \in G^m  : \  \mu_{G^k} (Z_y) =0  
			\right\}=\bigcap\limits_{n\ge 1} \left(G^m\setminus 
			Y_{\frac{1}{n}} \right)\] is $\omega$-internal and given  by a 
			countable intersection of internal subsets in $\mathcal F_m$. 
		\end{enumerate}
	\end{definition}
	
	\begin{remark}\label{R:tame_vs_language}
		The reader familiar with the notion of first-order formulae in a 
		given language containing the language of groups will 
		notice that a tame family $\mathcal F$ induces a first-order 
		language $\LL$ containing the language of groups after adding, 
		for every internal subset $X$ of $G^n$ in $\mathcal F$, a 
		distinguished predicate $R_X$ such that its interpretation in 
		$G$ equals $X$.  
  
Analogously, every language $\LL_0$ containing the language of groups can be enlarged to a language $\LL\supseteq \LL_0$ (as in \cite[Section 2.6]{eH12}) with $|\LL|\le \max (|\LL_0|,\aleph_0)$ such that $\LL$ induces a tame family of 
		internal 
		sets, by setting \[X_\psi=\{ (a_1,\ldots, a_n) \in G^n \ : \ 
		\psi(a_1,\ldots, a_n) \text{ holds in the $\LL$-structure } G \} \] 
		for every $\LL$-formula $\psi(x_1,\ldots, x_n)$. 
		
		With this translation in mind, we see that
		certain internal sets automatically belong to a tame 
		family $\mathcal F$, without having to write them explicitly as a 
		Boolean combination of suitable projections. For example, if the 
		subset $S$ of $G\times G$ is in $\mathcal F$ (and thus 
		\emph{definable}), then so is the set of BMZ-corners \[ \{ (a,b) \in G^2  : \ \text{ 
		for some $g$ in $G\setminus\{1_G\}$, all of }  (a,b), (g\cdot a, 
		b), (a, g\cdot b) \text{ lie in $S$}   \}.\]
	\end{remark}

	To obtain an example of a countable tame family of 
	internal sets it suffices to close 
	the family of internal sets given by group words under countably many instances of properties (II)-(IV). More generally, the following holds.
	
	\begin{fact}\label{F:TameNum}
		Given any countable family $\mathcal F_0$ of internal sets, there is 
		a countable tame family $\mathcal F$ of internal sets extending 
		$\mathcal F_0$.
	\end{fact}
	
	\begin{remark}
		By Fact \ref{F:TameNum}, whenever we want to apply 
		Lemma \ref{L:Transfer} to a distinguished internal set $X$ in a 
		non-standard finite group, we may always assume that $X$ belongs to a 
		countable tame family. 
	\end{remark}
	
	If the tame family $\mathcal F$ arises as the family of definable 
	sets with respect to a fixed language $\LL$ as explained in Remark 
	\ref{R:tame_vs_language}, then the following notion of richness 
	corresponds to the classical model-theoretic notion of an elementary 
	substructure in that particular language. 
	
	\begin{definition}\label{D:Rich}
		A subset $M$ of a non-standard finite group  $G$ is {\em rich with 
		respect to the tame family} $\Fc$ if 
		every non-empty fibre $X_{a}$ of $G^n$, where $a$ is tuple in 
		$M$ and $X$ is an internal subset of $G^{n+|a|}$ in $\Fc$, 
		contains an $n$-tuple whose coordinates all lie in $M$, {\it i.e.} 
		$X_a(M)=X_a\cap M^n\neq \emptyset$.
	\end{definition}
	
	A straightforward chain argument, mimicking the proof of 
	Downward 
	L\"owenheim-Skolem for first-order languages, yields the following 
	result.
	
	\begin{remark}[Downward L\"owenheim-Skolem]\label{R:LS}
		Given a countable tame family $\Fc$ and a countable subset $A$ of  a 
		non-standard finite group $G$, 
		there is a countable subset $M$ of $G$ containing $A$ which is rich 
		with respect to $\Fc$.
	\end{remark} 
	
	Notice that every rich subset is in particular an infinite subgroup of 
	$G$. 
	However, a countable rich subset $M$ no longer satisfies the  
	$\aleph_1$-saturation condition for covers of internal subsets. Indeed, the $M$-points of the internal set given by 
	$x=x$ 
	can be covered by countably many singletons (running through 
	every point in 
	$M$), yet it does not admit a finite cover. 
	\medskip
	
	\noindent \textbf{Henceforth, we fix a countable tame 
	family $\Fc$ of internal sets. All internal and $\omega$-internal 
	sets  we shall consider belong to this particular family. }
	\medskip

	Given a subset $A$ of a non-standard finite group $G$, we will 
	denote by $\Fc(A)$ the collection of all fibres $Y_{a}$ of $G^{n}$, 
	with $n$ running over all possible natural numbers, where $a$ is a 
	tuple in $A$ and $Y$ is an internal subset of $G^{n+|a|}$ in 
	$\Fc$. 
\begin{definition}\label{D:def_par}
  With the previous notation, the internal subsets in $\Fc(A)$ are said to be \emph{defined over $A$}. 

An $\omega$-internal set $X=\bigcap_{n\in \N} X_n$ is {\em defined over $A$} if each one of the internal sets $X_n$ is defined over the parameter set $A$.
\end{definition}

\begin{remark}\label{R:def_par}
Note that for every internal set $Z$ of $G^{k+m}$ and every rational number $q$ in $[0,1]$, there exists an internal set $Y_q$ as in Condition (IV) of Definition \ref{D:Tame} which is defined over the same parameters as $Z$. 
\end{remark}

	\begin{definition}\label{D:Type}
		Consider a non-standard finite group $G$ and a countable subset 
		$A$ of $G$. Given a tuple $b$ of elements of $G$, 
		its {\em type  over $A$} is the $\omega$-internal set 
		\[
		\tp(b/A) = \bigcap_{\mathclap{\substack{b\in X  \\ X\in 
		\Fc(A)}}}  X.
		\]
	\end{definition}
	Since the set $\Fc(A)$ is countable, by Remark 
	\ref{R:InfInt} (b) the type $\tp(b/A)$ is an uncountable  
	$\omega$-internal set whenever $b$ does not lie in a finite fibre 
	defined over $A$ of some internal set in $\Fc$ (that is, in model-theoretic terms, whenever the type is not algebraic). Note that $\tp(b'/A) 
	= \tp(b/A)$ if and only if $b'$ belongs to $\tp(b/A)$.  

\begin{remark}\label{R:aleph_1}
     Consider a countable subset 
		$A$ of parameters in a non-standard finite group $G$. Given two tuples $(b, c)$ and $b'$ such that $b'$ belongs to $\tp(b/A)$, there is some tuple $c'$ of length $k=|c|$ such that $(b', c')$ belongs to $\tp((b, c)/A)$ (denoted henceforth, as is standard, by $\tp(b,c/A)$ for brevity). Whilst for model-theorists this is an easy consequence of $\aleph_1$-saturation, we will include a quick proof for the sake of completeness. Indeed, by $\aleph_1$-saturation, it suffices to show that every finite intersection \[(Y_1 \cap \cdots \cap Y_n)_{b'}  = (Y_1)_{b'} \cap \cdots \cap (Y_n)_{b'} \subseteq G^{k}\] is non-empty, where each of the $Y_j$ varies over all internal sets in $\tp(b, c/A)$. Now, the intersection $\bigcap_{j=1}^n Y_j$ belongs to $\tp(b,c/A)$  and thus contains the tuple $(b, c)$. Hence, by Definition \ref{D:Tame}, the projection onto the first $|b|$ coordinates of $\bigcap_{j=1}^n Y_j$ is an internal set defined over $A$ which contains $b$. Since $b'$ belongs to $\tp(b/A)$, there exists some $d$ such that $(b', d)$ belongs to  $\bigcap_{j=1}^n Y_j$, as desired (note that the element $d$ need not be $c'$). 
\end{remark}	
	A classical Ramsey argument and the $\aleph_1$-saturation argument in Remark \ref{R:aleph_1} allows us to produce, out of a given infinite sequence and a countable set of 
	parameters, a new sequence with a remarkable property known as 
	\emph{indiscernibility}. 
	
	\begin{definition}
		We say that a sequence $(a_i)_{i\in \mathbb N}$ is 
		\emph{indiscernible over} the countable subset $A$ (or 
		\emph{$A$-indiscernible}) if for every natural number $n$ and 
		every increasing enumeration $i_0<\dots<i_{n-1}$ we have that 
		$\tp(a_0,\ldots, a_{n-1} / A )$ equals $\tp(a_{i_0},\ldots, 
		a_{i_{n-1}} / A )$,  that is, \[ (a_0,\ldots, a_{n-1}) \in X \ 
		\Leftrightarrow \ (a_{i_0},\ldots, a_{i_{n-1}}) \in X\] for every 
		internal subset $X$ in $\Fc(A)$ of the appropriate arity.
	\end{definition}

	We also define a notion of \emph{density} for a non-standard finite group.
	
	\begin{definition}\label{D:Dense}
		Let $G$ be a non-standard finite group. 
		\begin{itemize}
			\item 
			An $\omega$-internal subset is {\em dense} if it is not contained 
			in any internal set of density $0$ (see Definition 
			\ref{D:Measure}).
			\item 
			Given a countable subset $A$ of $G$ and a tuple $b$ of 
			elements of $G$, we say that $b$ is {\em dense over $A$} if the $\omega$-internal subset 
			$\tp(b/A)$ is dense. 
		\end{itemize}
	\end{definition}
	
	\begin{remark}\label{R:Dense}~
		\begin{enumerate}[(a)]
			\item An internal subset $X$ is dense if and only if it intersects 
			every internal subset of density $1$. Notice that the internal set $X$ is dense (seen as an $\omega$-internal subset) if 
			and only if it has positive density. However, an 
			$\omega$-internal set can be dense and yet  have density $0$ (with 
			respect to the extension of the Loeb measure to the 
			$\sigma$-algebra generated by all internal sets, see 
			\cite[Remark 16]{BT14}).
			\item If an $\omega$-internal set $X=\bigcap_{n\in \N} X_n$ is 
			given by a decreasing chain, then the set $X$ is dense if and 
			only if each internal set $X_n$ is. Indeed, if $X$ were not 
			dense, this would be witnessed by an internal set $Y$ of density $0$. Thus the 
			decreasing intersection \[\bigcap_{n\in \N} (X_n \setminus Y) 
			= \emptyset,\] which, by $\aleph_1$-saturation, would yield that 
			some $X_n\setminus Y$ must be empty.
			\item   If $b$ is dense over $A$, then so is every element in 
			$\tp(b/A)$.
			\item If $b$ is dense over $A\cup \{g\}$, then so are $b\cdot 
			g$ and $g\cdot b$.
		\end{enumerate}
	\end{remark}

	\begin{lemma}\label{L:ExtensionDense}
		Let $X$ be an $\omega$-internal dense subset of a 
		non-standard finite group $G$ defined over a countable subset 
		$A$. If  $B$ is a countable subset of parameters extending $A$, then there 
		exists some element in $X$ which is dense over $B$. 
		
		In particular, if $c$ is dense over $A$, then there is some $c'$ in 
		$\tp(c/A)$ which is dense over $B$.
	\end{lemma}
	\begin{proof}
		Write $X=\bigcap_{n\in \N} X_n$. By $\aleph_1$-saturation, it 
		suffices to show that 
		\[
		X_1\cap \dots\cap X_m  \cap (G\setminus Y_1) \cap \dots \cap 
		(G\setminus Y_m) \neq\emptyset
		\]
		for every $m\ge 1$, where $Y_1, \dots , Y_m$ are fibres of 
		internal sets in $\Fc$ defined over $B$ and of density 
		$0$. If  the above intersection were empty, 
		then $ 	X_1\cap \dots\cap X_m$ would have density $0$, and  
		$X$ 
		would not be dense. 
	\end{proof}

\begin{lemma}\label{L:Type_dense_coord}
Consider a countable set of parameters $A$ and  a finite tuple $(b,c)$
	in	a non-standard finite group $G$ such that $b$ is dense over $A\cup\{c\}$. If $(b',c')$ belongs to $\tp(b,c/A)$, then $b'$ is dense over $A\cup\{c'\}$. 
\end{lemma}
\begin{proof}
 Assume for a contradiction that  $b'$ is not dense over $A\cup\{c'\}$. By Remark \ref{R:Dense} there exists a fibre $X$ of an internal set defined over $A\cup\{c'\}$ of density $0$ and containing $b'$. Write $X=Z_{c'}$ for some internal set $Z\subset G^{k+m}$ defined over $A$, where $k=|b|=|b'|$ and $m=|c|=|c'|$. By Remark \ref{R:def_par} (after taking set-theoretic complements), there is for every natural number $n\ge 1$  an internal set $Y_{\frac{1}{n}}$ defined over $A$ such that
\[
			\left\{ y \in G^m  : \  \mu_{G^k} (Z_y) < \frac{1}{n}   \right\} \subseteq 
			Y _{\frac{1}{n}} \subseteq \left\{ y \in G^m  : \  \mu_{G^k} (Z_y)\le \frac{1}{n}  
			\right\}.
			\]
Now the internal set $\widetilde Z_n=Z\cap \big(G^{k}\times Y_{\frac{1}{n}}\big)$ is defined over $A$ for every natural number $n\ge 1$ and contains $(b',c')$, so $(b,c)$ also belongs to the intersection of all the $\widetilde Z_n$. We conclude that $b$ is not dense over $A\cup\{c\}$, witnessed by the fiber $\widetilde{Z}_{c}$, as desired.  
\end{proof} 
	
Model-theoretically, it is convenient to capture the global 
behaviour of an internal (or definable) set in terms of a suitable 
(dense) element in the set. An example of this translation is the 
following result, which follows from the Fubini-Tonelli Theorem and 
 Caratheodory's Extension Theorem for the Loeb measure.
	
	\begin{fact} \cite[Lemma 1.10 as well as Remarks 1.13 \& 1.14 and Lemma 1.15]{MPP21} 
	\label{F:UdiExercise}
		Let $G$ be a non-standard finite group and $X$ be an $\omega$-internal 
		subset of $G^n$ defined over a countable subset $A$. The 
		following are equivalent.
		\begin{enumerate}[(a)]
			\item There exists some $(b_1,\ldots,b_n)$ in $X$ \emph{in 
				good position} over $A$, that is, each $b_i$ is dense over 
			$A\cup\{b_j:j<i\}$. 
			\item The set $X$ is dense.
		\end{enumerate}
		
		Furthermore, every tuple as in (a) is dense over $A$ 
		with respect to the Loeb measure on $G^n$. 
	\end{fact}	
	Note that if a tuple is in good position, then so is every subtuple. 
	\begin{definition}\label{D:Comp}
		Two internal subsets $X$ and $Y$ of $G^n$  are {\em 
		comparable} if the set $X\triangle Y$ has Loeb measure $0$.
	\end{definition}
	
	\begin{remark}\label{R:Comp} 
		By Lemma \ref{L:ExtensionDense}, two internal sets 
		$X$ and $Y$, both defined over $A$, are comparable if and only if 
		they contain the same dense elements over $A$.
		
		By the Fubini-Tonelli Theorem, two internal subsets $X$ and $Y$ of 
		$G^{m+n}$ are comparable if and only if for 
		$\mu_{G^n}$-almost all elements $b$ in $G^n$, the fibres $X_b$ 
		and $Y_b$ are comparable.
	\end{remark}
	
	\section{Stability and almost sure stability}\label{S:stab}
	
	Let $k\ge 1$ be a natural number. A \emph{half-graph of height 
	$k$ induced by} the relation $S \subseteq  X\times Y$ consists of a 
	sequence $(a_1,b_1,\ldots,a_k,b_k)$ with $a_i$ in $X$ 
	and $b_i$ in $Y$ for each $i=1,\dots, k$, such that the pair $(a_i,b_j)$ belongs to $S$ 
	if and only if $i\le j$. We denote by $\HH_k(S)$ the collection of all 
	half-graphs of height $k$ induced by $S$. Note that for each 
	sequence $(a_1,b_1,\ldots,a_k,b_k)$ in $\HH_k(S)$, the elements 
	$a_i$ must be pairwise distinct and similarly for the elements $b_i$.  
	
	\begin{definition}\label{D:stab}
		The relation $S \subseteq  X\times Y$ is \emph{$k$-stable} if it 
		induces no half-graph of height $k$, or equivalently, if 
		$\HH_k(S)=\emptyset$. A subset $A$ of a group $G$ is 
		\emph{$k$-stable} if the relation given by the Cayley graph of 
		$A$ in $G$ \[ \mathrm{Cay}(G,A)=\{ (g,h) \in G\times G  : \   
		g^{-1} \cdot h \in A\}\] is $k$-stable, or equivalently, if the 
		relation 
		\[ \Gamma(G,A)=\{ (g,h) \in G\times G  : \  h\inv \cdot g \in 
		A\}\] is $k$-stable.  
	\end{definition}
	
	\begin{example}\label{E:stable_examples}
		A non-empty subset $A$ of a group $G$ is $2$-stable if and only 
		if it is a coset of a subgroup of $G$.  Indeed, cosets 
  are clearly $2$-stable. To verify the other direction, it suffices to show that $a\cdot b\inv\cdot c$ belongs to $A$ for every $a$, $b$ and $c$ in $A$. Sidon subsets of abelian 
		groups, such as $2^{\mathbb N}$ in $\Z$, are easily seen to be 
		$3$-stable \cite[Lemma 1.3]{tS20}.
	\end{example}
	
	\begin{remark}\label{R:stab_ultrap}
		Consider a non-standard finite group $G=\prod_{n\to \Ult} G_n$ 
		and  an internal relation $S$ on $G\times G$ defined over a
		countable subset $A$.  By \L o\'s' Theorem (see Remark 
		\ref{R:internal} (e)), the relation $S$ is $k$-stable if and only if $S(G_n)$ is $k$-stable in $G_n$ for $\Ult$-almost all $n$ in $\N$. 
		Furthermore, a straightforward $\aleph_1$-saturation argument 
		yields that the internal relation $S$ is $k$-stable for some $k\ge 
		2$ if and only if there is no $A$-indiscernible sequence 
		$(a_i,b_i)_{i\in\N}$ such that the pair $(a_i,b_j)$ belongs to $S$ 
		if and only if $i\le j$. 
	\end{remark}
	
	The previous remark motivates the following definition.
	\begin{definition}\label{D:stab_infinite}
		Consider a non-standard finite group $G=\prod_{n\to \Ult} G_n$. 
		An $\omega$-internal relation  $S$  on $G\times G$ defined over 
		the countable subset $A$ is \emph{stable} if there is no infinite 
		$A$-indiscernible sequence $(a_i,b_i)_{i\in\N}$ such that the 
		pair $(a_i,b_j)$ belongs to $S$ if and only if $i\le j$.  
	\end{definition}
	Using the Krein-Milman theorem on the 
	locally compact Hausdorff topological real vector space of all 
	$\sigma$-additive probability measures, Hrushovski  \cite[Proposition 2.25]{eH12}
	proved the following result. Roughly speaking, it asserts that the 
	relation $R(a,b)$ defined by requiring the measure of the intersection of two associated internal sets 
	$X_a$ and $Y_b$ to exceed a certain threshold
 is stable. 
	For the
	purpose of this work, we will state an adapted version of
	\cite[Lemma 2.10]{eH12}, extracting it from the formulation in
	\cite[Fact 2.2 \& Corollary 2.3]{MPP20}.
	\begin{fact}\label{F:Measure_eq}
		Let $G$ be a non-standard finite group and let $\alpha$ be a real 
		number in $[0,1]$. 
		\begin{enumerate}[(a)]
			\item Given two internal subsets $X$ of $G^{n+r}$ and $Y$ of 
			$G^{n+s}$ in $\Fc$, the relation $R^\alpha_{X,Y}$ of 
			$G^r\times G^s$ defined by  \[ R^\alpha_{X,Y}(a,b) \ 
			\Leftrightarrow \ \mu_{G^n}(X_a \cap Y_b) \le \alpha \] is 
			stable. Notice that  $R^\alpha_{X,Y}$ is $\omega$-internal by 
			Definition \ref{D:Tame} (IV), and definable over the same parameters needed to define both $X$ and $Y$, by Remark \ref{R:def_par}. 
			\item If the  two  subsets $X$ of $G^{n+r}$ and $Y$ of 
			$G^{n+s}$ are $\omega$-internal and defined over $A$, the 
			relation  $R_{X,Y}$ of $G^r\times G^s$ defined by \[ 
			R_{X,Y}(a,b) \ \Leftrightarrow \ \text{ the $\omega$-internal 
			subset $X_a \cap Y_b$ of $G^n$  is not dense}\] is 
			\emph{equational}, that is, there is no  $A$-indiscernible sequence 
			$(a_i, b_i)_{i\in \N}$ such that $X_{a_0} \cap Y_{b_0}$ is dense yet $X_{a_0} \cap Y_{b_1}$ is not.
			
			\item Equationality implies the following (see \cite[Remark 2.1 
			\& Corollary 2.3]{MPP20}). For every $\omega$-internal subset 
			$X$ of $G^{n+r}$ defined over a countable rich subset $M$, if 
			the fibre $X_a$ is dense for some $a$ which is itself dense over $M$, then so 
			is the intersection $\bigcap_{i=0}^m X_{b_i}$ for every $m$ in 
			$\N$ and every tuple $(b_0, \ldots, b_m)$ of elements in 
			$\tp(a/M)$ in good position over $M$, as in Fact 
			\ref{F:UdiExercise} (a), that is, the element 
			$b_i$ is dense over $M\cup\{b_j:j<i\}$ for $i\le m$.  
		\end{enumerate}
	\end{fact}

While the notion of stability has proved highly profitable in model theory, from a more combinatorial perspective it is open to the criticism that whether or not a relation is stable is very sensitive to minor 
	perturbations. In particular, modifications by an internal set of 
	measure $0$ can 
	destroy stability, as the following example illustrates. 
	\begin{example}\label{E:destroy__stab}
		For some non-principal ultrafilter $\Ult$, consider the 
		non-standard finite group $G=\prod_{n\to \Ult} \Z/n^2\Z$  as 
		well as the internal set $X=\prod_{n\to \Ult} \{\overline 
		0,\ldots, \overline{n-1}\}$ and the internal relation $S$ on 
		$X\times X$ arising from the standard linear order of $\Z$ 
		restricted to the set of representatives $\{0,\ldots, n-1\}$. By \L 
		o\'s's theorem (see Remark \ref{R:internal} (e)), the collection 
		$\HH_k(S)$ of induced half-graphs of height $k$ is non-empty for 
		every natural number $k\ge 1$, so the internal relation $S$ is never 
		stable. However, for every index $k\ge 1$, the internal set 
		$\HH_k(S)$ has $\mu_{G^{2k}}$-measure $0$, so it is 
		comparable to the empty relation, which in turn is $k$-stable.
	\end{example}
	Motivated by the above example, we introduce a weakening of 
	stability which is combinatorially robust with respect to 
	perturbations by sets of measure $0$.
	
	\begin{definition}\label{D:robust}
		Consider a non-standard finite group $G$ as well as two internal 
		subsets $X$ of $G^{n}$ and $Y$ of $G^{m}$. Given a natural number $k\geq 1$, an internal 
		relation $S \subseteq  X\times Y$ is \emph{almost surely $k$-stable} 
 if $\mu_{G^{(n+m)k}}(\HH_k(S))=0$, 
		that is,  the collection of all half-graphs of height $k$ induced by $S$ is 
		negligible in $G^{(n+m)k}$ with respect to the non-standard 
		counting measure.  
	\end{definition}
	
	\begin{remark}\label{R:robust}~
		\begin{enumerate}[(a)]
			\item 	It is immediate that almost sure stability is preserved 
			by 
			finite 
			Boolean combinations as well as by permutation of the 
			variables: 
			if $S\subseteq X\times Y$ is almost surely $k$-stable, then so is the 
			inverse relation \[S^{opp}=\{(y,x)\in Y\times X  : \  (x,y) \in 
			S\}.\]
			
			\item 	Fact \ref{F:UdiExercise} yields that the internal 
			relation 
			$S$ defined over the countable subset $A$ is not almost surely
			$k$-stable if and only if it induces a half-graph of height $k$ 
			witnessed by a dense tuple $(a_1, b_1, \ldots, a_k, b_k)$ in 
			good 
			position over $A$. 
			
			As a consequence, whenever a relation $S$ is comparable to an 
			internal $k$-stable or even an almost surely $k$-stable relation $S'$, 
			Remark \ref{R:Comp} yields 
			that $S$ is  almost surely $k$-stable. 
		\end{enumerate}
	\end{remark}
	
	We will see in Proposition \ref{P:robuststable_boxes} below that every 
	almost surely stable relation $S\subseteq X\times Y$ can be 
	approximated for every $\epsilon>0$ by a finite union of  internal 
	boxes of the form $X'\times Y'$. Hence, the almost surely $k$-stable 
	relation is $\epsilon$-comparable to a stable relation $S'$, although 
	our methods do not allow us to compute the degree of stability of 
	$S'$. 
	
	 In fact,  the (bipartite version of the) induced removal lemma  
	\cite{EFR86} implies that every almost surely $k$-stable must be 
	comparable to a $k$-stable relation, but we do not know how to give 
	a model-theoretic proof of this result. So we ask the following: 
	
	\begin{question}
		Is there a model-theoretic account that every almost surely $k$-stable 
		internal relation is comparable (or $\epsilon$-comparable) to an internal 
		$k$-stable relation?
	\end{question}
A straightforward application of Lemma 
\ref{L:Transfer} (with $r=0$) yields an analogue of almost sure stability for families 
of 
finite groups  
in terms of  their asymptotic behaviour: 
\begin{remark}\label{R:robust_asympt}
Consider a family $(G_\ell, X_\ell, Y_\ell, S_\ell)_{\ell\in\N}$, 
 where for each $\ell\in \N$, $G_\ell$ is a finite group,
$X_\ell\subseteq G_\ell^n$ and $Y_\ell\subseteq G_\ell^m$ for some $n,m\in \N$, and $S_\ell\subseteq X_\ell\times 
Y_\ell$.  The following are equivalent for 
every natural number $k\ge 1$.
\begin{enumerate}[(a)]
	\item For every  $\theta>0$ there is some 
	$\ell_0=\ell_0(k,\theta)$ in 
	$\N$ such that $|\HH_k(S_\ell)| \le \theta |G_\ell|^{k(n+m)}$ for all 
	$\ell\ge \ell_0$, where $\HH_k(S_\ell)$ is the collection of all 
	half-graphs 
	of height $k$ induced by $S_\ell$ on the finite group $G_\ell$.
	\item For every non-principal ultrafilter $\Ult$ on $\N$, the 
	internal relation $S=\prod_{n\to\Ult} S_\ell$ is almost surely $k$-stable.
\end{enumerate}
If  (a) holds, we say that  $(G_\ell, X_\ell, Y_\ell, S_\ell)_{\ell\in\N}$ 
is an \emph{almost surely $k$-stable family}. 
\end{remark}

	A key feature of stable relations is that they are \emph{stationary} \cite[Lemma 2.3]{eH12}, in the sense that their 
	truth value is constant along
	the set of pairs of realisations which are \emph{non-forking independent}.
	Non-forking independence is a fundamental notion in model 
	theory, originally due to Shelah, defined in combinatorial terms for 
	any structure \cite[Chapter III.1, Definition 1.4]{sSbook}. We
	will not need to introduce non-forking independence in this paper. Instead, inspired by classical results in model theory, we will show that the truth value of an almost surely
	stable relation remains constant along the set of pairs of realisations for 
	which one coordinate is dense over the other. 
	
	Before stating the corresponding result, we introduce some 
	notation. 
	\begin{definition}\label{D:goodpos}
		Fix two types $\tp(a/M)$ and $\tp(b/M)$ over a countable rich 
		subset $M$. We denote by $\GP(\tp(a/M),\tp(b/M))$  the set of 
		all 
		pairs $(a',b')$ with $a'$ 
		in $\tp(a/M)$ and $b'$ in $\tp(b/M)$ such that $(a',b')$ or 
		$(b',a')$ is in good position over $M$ (see Fact 
		\ref{F:UdiExercise}).  
	\end{definition}
	Note that the set $\GP(\tp(a/M),\tp(b/M))$ is empty exactly if 
	one of the types $\tp(a/M)$ or $\tp(b/M)$ is not dense, by Lemma 
	\ref{L:ExtensionDense}. Moreover:
	\[
	(a',b')\in \GP(\tp(a/M),\tp(b/M)) \ \Leftrightarrow \ (b',a')\in 
	\GP(\tp(b/M),\tp(a/M)).
	\]
	
	\begin{theorem}\label{T:robuststable_main}
		Consider a non-standard finite group $G$ as well as two internal 
		subsets $X$ of $G^{n}$ and $Y$ of $G^{m}$ and an internal 
		relation $S$ on $X\times Y$,  all defined over a countable rich 
		subset $M$. 
		
		If  $S$ is almost surely $k$-stable for some $k\ge 1$, then for every two (dense)
		types $\tp(a/M)$ and $\tp(b/M)$, the set 
		$\GP(\tp(a/M),\tp(b/M))$ is either 
contained in $S$ or disjoint from $S$. 
	\end{theorem}
	
	\begin{proof}
		If $a$ does not lie in $X$, neither does any element of 
		$\tp(a/M)$, and likewise if $b$ does not lie in $Y$. In this case 
		$X\times Y$ (and thus $S$) is clearly disjoint from 
		$\GP(\tp(a/M),\tp(b/M))$. So we 
		may assume that $\GP(\tp(a/M),\tp(b/M))$ is a subset of 
		$X\times Y$. 
		
		Suppose now that $S$ is not  disjoint from   
		$\GP(\tp(a/M),\tp(b/M))$.  Up to relabelling, we may assume by 
		Remark \ref{R:robust}(a) that $(a,b)$ lies in $S$, with $(a, b)$ 
		in 
		good position, that is, the element $b$ is dense over 
		$M\cup\{a\}$ (keeping in mind that $a$ is already dense over 
		$M$).
		
		\begin{claim*}
			All pairs $(c, d)$ in $\GP(\tp(a/M),\tp(b/M))$ with $(d, c)$ 
			in good position over $M$ are contained in $S$. 
		\end{claim*}		
		\begin{claimproof*}
			Suppose to the contrary that there is a pair $(c, d)$ with $c$ dense 
			over 
			$M\cup\{d\}$ not contained in $S$. By Remark \ref{R:aleph_1} 
			(using the fact that $d$ belongs to the type $\tp(b/M)$),  there is
			a pair 
			$(a',b)$ in $\tp(c, d/M)$. So the original tuple $(a, b)$ belongs to $S$ but  $(a', b)$ does 
			not lie 
			in $S$ (for $S$ is definable over $M$). 
			In particular, we have two pairs $(a, b)$ and $(a', b)$ satisfying
			the following:
			\begin{itemize}
				\item the type $\tp(a/M)=\tp(a'/M)$ is dense;
				\item the element $b$ is dense over $M \cup\{a\}$ and  
				$a'$ is dense 
				over $M \cup\{b\}$ (by Lemma \ref{L:Type_dense_coord});
				\item the pair $(a, b)$ belongs to $S$ but $(a', b)$ does not.
			\end{itemize}
			We will now construct inductively, for every $n$ in $\N$, a 
			dense 
			sequence $(a_0, b_0, \ldots, a_n, b_n)_{0\le i\le n}$ in good 
			position 
			over $M$ witnessing that $S$ induces a half-graph of height 
			$n$ 
			such that  $a_i$ belongs to $\tp(a/M)$ and $b_i$ to 
			$\tp(b/M)$ for 
			each $i\le n$. In particular, the case $n=k$ contradicts the almost sure 
			stability of $S$ by Remark \ref{R:robust} (b). 
			
			For $n=0$, set $a_0=a$ and $b_0=b$, so $(a_0, b_0)$ is in good 
			position over $M$ and lies in $S$, as desired. Suppose now that 
			the
			sequence $(a_0, b_0, \ldots, a_r, b_r)$ has already been 
			constructed 
			for some $r< n$. Consider the $\omega$-internal set \[Z =\{ 
			(x,y) \in 
			X\times Y  : \  x\in \tp(a/M) \ \text{and} \ (x,y) \notin S \},\] 
			which is 
			defined 
			over $M$. Clearly $(a', b)$ belongs to $Z$, so the fibre 
			$Z_b\subseteq 
			X$ is dense. 
			By Fact \ref{F:Measure_eq} (c), so is the intersection $ 
			Z_{b_0}\cap\cdots \cap Z_{b_r}$, since the subtuple 
			$(b_0,\ldots, 
			b_r)$ is also in good position over $M$ and all  coordinates lie 
			in 
			$\tp(b/M)$. Hence, there is an element $a_{r+1}$ in the above 
			intersection which is dense over $M\cup\{a_i, b_i\}_{0\le i\le 
			r}$. In 
			particular, the element $a_{r+1}$ lies in $\tp(a/M)$ but 
			$(a_{r+1}, 
			b_i)$ does not lie in $S$ for $i\le r$. 
			
			Similarly, the $\omega$-internal set \[W =\{ (x,y) \in 
			X\times Y : \  y\in \tp(b/M) \ \text{and}\ (x,y) \in S \}\]  is 
			also defined 
			over $M$ and the fibre $W_a\subseteq Y$ is dense, since $(a, 
			b)$ lies 
			in 
			$W$. Thus, there is an element $b_{r+1}$ in $\tp(b/M)$ 
			contained in 
			$W_{a_0}\cap\cdots\cap W_{a_{r+1}}$ dense over $M\cup 
			\{a_i, 
			b_i\}_{0\le i\le r} \cup \{a_{r+1}\}$, which yields that the 
			tuple \[ 
			(a_0, b_0, \ldots, a_{r+1}, b_{r+1})\] is in good position over 
			$M$ by Fact 
			\ref{F:UdiExercise} and satisfies the desired properties. 
		\end{claimproof*}
		
		By the previous claim, choose any pair $(a',b')$ in $S\cap 
		\GP(\tp(a/M), \tp(b/M))$ such that $a'$ is dense over 
		$M\cup\{b'\}$. The pair $(b',a')$ is in good position over $M$ 
		and it belongs to the inverse relation $S^{\rm opp}$, which is 
		again almost surely $k$-stable, by Remark \ref{R:robust} (a). 
		We conclude by the Claim (applied to $S^{\rm opp}$ inverting 
		the roles 
		of $X$ and $Y$) that all pairs $(b'', a'')$ in $\GP(\tp(b/M), 
		\tp(a/M))$ 
		with $b''$ dense over $M\cup\{a''\}$ must be contained 
		in $S^{\rm opp}$. Hence, all pairs $(a'', b'')$ in $\GP(\tp(a/M), 
		\tp(b/M))$ in good position over $M$ belong to $S$, so the set 
		$\GP(\tp(a/M), \tp(b/M))$  is fully contained in $S$, as desired. 
	\end{proof}
	
	\begin{cor}\label{C:keylemma}
		Consider a non-standard finite group $G$ as well as two internal 
		subsets $X$ of $G^{n}$ and $Y$ of $G^{m}$ and an internal 
		relation $S$ on $X\times Y$,  all defined over a countable rich 
		subset $M$.  If  $S$ is almost surely $k$-stable for some $k\ge 1$, 
		then the fibres $S_a$ and $S_{a'}$ are comparable 
		whenever $a'$ belongs to the dense type $\tp(a/M)$.
	\end{cor}
	\begin{proof}
		If the set-theoretic difference $S_a \setminus S_{a'}$ were 
		dense for some choice $a$ and $a'$, both in the same dense type 
		over $M$, then it would contain some element $b$ which is dense over 
		$M\cup\{a, a '\}$.  In particular, both $(a, b)$ and $(a', b)$ would lie in 
		$\GP(\tp(a/M), \tp(b/M))$, yet $(a,b)$ belongs to $S$ but $(a', 
		b)$ does not, which contradicts Theorem 
		\ref{T:robuststable_main}. 
	\end{proof}
	
	\begin{prop}\label{P:robuststable_boxes}
		Consider a non-standard finite group $G$ as well as two internal 
		subsets $X$ of $G^{n}$ and $Y$ of $G^{m}$ and an internal 
		relation $S$ on $X\times Y$,  all defined over a countable rich 
		subset $M$.  If  $S$ is almost surely $k$-stable for some $k\ge 1$, 
		then for every $\epsilon >0$ there is a finite (possibly empty) 
		union 
		of boxes \[ U=\bigcup\limits_{i=1}^\ell  (X_i\times Y_i),\] with all 
		$X_i\times Y_i \subseteq X\times Y$,  
		defined over $M$, such that $\mu(U\triangle S) < \epsilon$.
	\end{prop}
	Note that every box is $2$-stable and thus the finite union $U$ is 
	$r$-stable for some $r$ which can be explicitly computed in 
	terms of the number $\ell$ of boxes occurring in $U$. However, at the time of writing we 
	are unable to give an explicit bound on $\ell$ 
	or show that $U$ and $S$ are actually comparable. 
	\begin{proof}
		If $S$ is not dense in $G^{n+m}$, then it suffices to take $U$ to be the empty union of boxes. Thus, we may assume that the 
		internal subset 
		$S$ of $G^{n+m}$ is dense. 
		
		\begin{claim}\label{Claim:family}
			Consider an arbitrary subset $Z$ of $S$ in the $\sigma$-algebra of 
			internal sets defined over $M$. (Note that $Z$ need not be 
			$\omega$-internal, but it is measurable with respect to the 
			extension of the Loeb measure.) If $Z$ has positive measure, 
			then there exists a dense element $a$ in $G^n$ such that the fibre $Z_a$ is dense and satisfies the following property ($\star$): for every internal 
			subset $W$ of $Y$ defined over $M$, if 
			$\mu_{G^m}(W\setminus  
			S_a)=0$, 
			then \[ \mu_{G^n}\big( \{x\in G^n  : \  
			\mu_{G^m}(W\setminus S_x)=0\} \big) > 0.\]  
		\end{claim}
		\begin{claimproof}
			
			By Fubini-Tonelli (Remark \ref{R:Fubini}), the measurable set 
			\[Z_1=\{ x\in G^n  : \  \mu_{G^m}(Z_x)>0\}\] has positive 
			measure (since $\mu_{G^{n+m}} (Z)$ is not zero). For every 
			internal subset $W$ of $Y$ defined over $M$, set \[ D_W=\{ 
			x\in G^n 
			 : \  \mu_{G^m}(W \setminus S_x)=0 \}.\]  Notice that there 
			are only 
			countably many such sets $D_W$. Consider the 
			subfamily 
			$\mathcal D$ of all sets $D_W$ of measure $0$. 
			Similarly, let $\mathcal E$ be the countable family of all 
			internal subsets $X'$ of $G^n$ defined over $M$ of measure 
			$0$. By $\sigma$-additivity of the (extension of the) Loeb 
			measure, the set 
			\[
			Z_1 \setminus \Big( \bigcup_{D_W\in\mathcal D} D_W \ \cup \ 
			\bigcup_{X'\in \mathcal E} X' \Big) 
			\]  has positive measure, so it is non-empty. Now choose some 
			element $a$ in the above set-theoretic difference. By 
			the choice of $\mathcal E$, the type $\tp(a/M)$ is dense and so is 
			the fibre $Z_a$ (since $a$ lies in $Z_1$).  Furthermore, the 
			element $a$ satisfies ($\star$)
			by our choice of the family $\mathcal D$. 
		\end{claimproof}
	As the internal set $S$ is dense, it has positive measure. By the previous 
	claim, there is an element $a$ which is dense over $M$  satisfying 
	$(\star)$ with respect to $Z=S$ such that the fibre $S_a$ is dense. Thus, 
	the fibre $S_a$ contains an element $b$ which is dense over $M\cup\{a\}$. 
	By Fact \ref{F:UdiExercise}, the dense pair $(a,b)$ is in 
	good position over $M$ and clearly lies in $S$.

		\begin{claim}\label{Claim:specialpt}
			For every pair $(a, b)$ in $S$ in good position over $M$ such that 
			$a$ satisfies 
			($\star$), there are an internal subset $Y'$ of $Y$ and an 
			$\omega$-internal subset $X'$ of $X$,  both defined 
			over $M$, such that $(a,b)$ lies in $X'\times Y'$ and $\mu((X'\times 
			Y')\setminus S)=0$. 
		\end{claim}
		\begin{claimproof}
			By Theorem \ref{T:robuststable_main},  there is no element in 
			$\tp(b/M)$ dense over $M\cup\{a\}$ outside of $S_a$, so  
			the $\omega$-internal subset $\tp(b/M) \setminus S_a$ of $Y$ 
			is not dense. By $\aleph_1$-saturation, there is some internal 
			subset $Y'$ of $Y$ defined over $M$ containing $b$ with 
			$\mu_{G^m}(Y'\setminus S_a)=0$. We 
			deduce  by $(\star)$ that $\mu_{G^n}(X')>0$, where 
			\[ X'= \{x \in X  : 
			\ \mu_{G^m}(Y'\setminus S_x)=0\}\] is an $\omega$-internal subset 
			(by Definition \ref{D:Tame} (IV)) containing $a$ (this set is 
			$D_{Y'}$ with $Z=S$ in the notation of Claim 
			\ref{Claim:family}). 
			
			A 
			straightforward computation using Fubini-Tonelli yields 
			\[
			\mu_{G^{n+m}} ((X'\times Y')\setminus S) = \int_{X'} 
			\mu_{G^m}(Y'\setminus S_x) \, \mathrm{d}\mu_{G^n} = 0, \]
			as desired. 
		\end{claimproof}
		
		Consider now the countable collection $\mathcal B_S$ of all 
		subsets 
		$X'\times Y'$ as in Claim \ref{Claim:specialpt}, that is, the 
		set $Y'$ is internal whilst $X'$ is $\omega$-internal, both are
		defined over $M$ and 	\[ X'= \{x \in X : 
		\ \mu_{G^m}(Y'\setminus S_x)=0\},\] so   
		$\mu_{G^{n+m}}((X'\times Y')\setminus 
		S)=0$. By 
		$\sigma$-additivity, we have that \[ \Big(\bigcup_{X'\times 
		Y'\in \mathcal B_S} (X'\times Y') \Big)\setminus S \] has measure 
		$0$ (with respect to the extension of the Loeb measure). Observe further that the set \[ Z=S\setminus \bigcup_{X'\times Y'\in \mathcal 
		B_S} (X'\times Y') \]
		belongs to the $\sigma$-algebra of internal sets defined over 
		$M$.  If $Z$ has positive measure, we deduce as in the discussion right 
		after Claim \ref{Claim:family} that $Z$ contains a pair $(a,b)$ in good 
		position over $M$ with $a$ satisfying 
		$(\star)$. It follows from Claim 
		\ref{Claim:specialpt} that $(a,b)$ lies in some 
		box $X'\times Y'$ in $\mathcal B_S$, contradicting the choice of  
		$Z$.  Thus, the sets $S$ and $\bigcup_{X'\times 
		Y'\in \mathcal B_S} (X'\times Y')$ are comparable  with respect to 
		the extension of the Loeb measure. We need only show that we 
		can replace the latter union, modulo an $\epsilon$-error, by 
		an actual finite union of internal 	boxes defined over $M$.
		
		Given $\epsilon>0$, by continuity of the Loeb measure there is some 
		integer $\ell$ 
		and subsets $X'_1\times 
		Y_1,\ldots, X'_\ell\times Y_\ell$ in $\mathcal B_S$ such that \[ 
		\mu_{G^{n+m}}( 
		S\setminus 
		\bigcup\limits_{i=1}^\ell (X'_i\times Y_i))<\frac{\epsilon}{2}.\] Note 
		that $\ell>0$, since $\mathcal B_S$ is non-empty, for $S$ is assumed to 
		be dense.   
		For each $1\le i\le \ell$, the $\omega$-internal subset 
		$(X'_i\times Y_i)\setminus 
		S$ has measure $0$, so we can find internal subsets $X_i$, with each $X_i$ containing $X'_i$, such that 
		\[\mu_{G^{n+m}}\big((X_i\times Y_i)\setminus S\big) 
		<\frac{\epsilon}{2\ell}.\]
		Note that $X_i$ is again internal and defined over $M$. Now set 
		$U=\bigcup_{i=1}^\ell (X_i\times Y_i)$, so $U$ 
		contains $\bigcup_{i=1}^\ell (X'_i\times Y_i)$, and it follows
		that 
		 \begin{align*}
		\mu_{G^{n+m}}\left(S \triangle U\right) & = \mu_{G^{n+m}} (S 
		\setminus U) + \mu_{G^{n+m}}(U\setminus S)  \\ & \le \mu_{G^{n+m}} 
		\big(S 
		\setminus \bigcup_{i=1}^\ell (X_i'\times Y_i) \big)  + \sum_{i=1}^{\ell}
		\mu_{G^{n+m}}\left((X_i\times Y_i)\setminus S \right) \\ & < 
		\frac{\epsilon}{2} + 
		\ell\cdot \frac{\epsilon}{2\ell} = \epsilon.
		\end{align*}
		This completes the proof of Proposition \ref{P:robuststable_boxes}.
	\end{proof}
	\section{Corners and squares}\label{S:corner}
	
Observe that the results in Section \ref{S:stab} did not make use of the underlying 
	group structure. In this section, on the other hand, the group structure will play a 
	fundamental role in guaranteeing the existence of certain 2-dimensional patterns in a dense 
	almost surely stable relation. 
	
	Given a countable rich subset $M$,  the group $G(M)=G\cap 
	M$ of 
	$M$-rational points of $G$ naturally acts on the collection of  types 
	over $M$ by \[  \tp(a/M) \star g = \tp(a\cdot g/M) .\] This is a right 
	action, but there is also a natural left action, defined analogously. With respect to the previous right action, we can thus consider the stabilizer in $G(M)$ of a type. This subgroup need not be definable in general. However, in the presence of stability, the stabilizer of a type becomes $\omega$-internal 	(or type-definable), though we will not need this for the purpose of this article.

	Since the seminal work of Kim and 
		Pillay \cite{KP97} on simple theories, many notions and techniques from 
	geometric stability  
	have been adapted successfully to other contexts. Pushing beyond simplicity, in groundbreaking work \cite{eH12}
	Hrushovski established the existence of an $\omega$-internal 
	subgroup of a non-standard finite group that plays 
	the role of the 
	\emph{stabilizer} of \emph{every} dense type over the countable rich substructure $M$.
	This subgroup is known in model theory as the \emph{connected 
	component} of $G$ over $M$ and denoted by $\GO M$.
	
%
	The next fact summarises the content of Hrushovski's \emph{stabiliser theorem}, tailored to our particular context. For ease of reference, the presentation largely follows the formulation of \cite[Theorem 2.12]{MOS18}. 
	
	\begin{fact}\textup{(}\cite[Theorem 3.5]{eH12} \& \cite[Theorem 2.12]{MOS18}\textup{)} 
	\label{F:G00}
		Let  $M$ be a countable rich subset of a non-standard finite group 
		$G$. 
		Then there exists an $\omega$-internal normal subgroup $\GO M$
		defined over $M$ with the following properties.
		\begin{enumerate}[(a)]
			\item\label{I:nbhood}  The $\omega$-internal subset $\GO M$ equals a countable 
			intersection of internal \emph{generic symmetric neighborhoods} 
			$Z_n$ of 
			the 
			identity, each defined over $M$, that is,
			\begin{itemize}
	\item  each $Z_n=Z_n\inv$ is symmetric and contains the identity element 
	$1_G$;
	\item finitely many translates of each $Z_n$ cover the group $G$;
	\item $Z_{n+1}\cdot Z_{n+1}\subseteq Z_n$ for all $n$ in $\N$. 
			\end{itemize}
		
			\item\label{I:diff} Whenever 
			$(a, b)$ is in good position over $M$ with $b$ in $\tp(a/M)$,
			the dense element $b\cdot 
			a\inv$ over $M$ belongs to $\GO M$. In particular, the 
			$\omega$-internal set $\GO M$ is dense.  
			\item\label{I:action}  Whenever $a$ is dense over $M$ and $g$ in $\GO M$ is 
			dense over $M$,  
			the $\omega$-internal set \[ \tp(a/M)\cap  \tp(a/M)\cdot g\] is 
			again dense,  that 
			is, there exists some $a'$  in $\tp(a/M)$ which is dense over 
			$M\cup\{g\}$ 
			such that $a'\cdot g$ belongs to $\tp(a/M)$.
			
			Similarly, the $\omega$-internal set  $\tp(a/M)\cap  g\cdot 
			\tp(a/M)$ is 		also dense,  so there exists some $a''$  in 
			$\tp(a/M)$ which is dense over 		$M\cup\{g\}$ 
			such that $g\cdot a''$ belongs to $\tp(a/M)$. 
			\item\label{I:GNsbgpGM} If $N$ is a countable rich subset of $G$ containing $M$, 
			then $\GO 
			N$ is a 
			subgroup of $\GO M$. 
			\end{enumerate}
		Furthermore, it follows that if 
		$G$ has finite exponent, then $\GO M$ is a countable 	
		intersection of internal subgroups, each defined over $M$ and 
		of finite index: Indeed,  we can equip  the quotient group $G/ \GO 
		M$ with the so-called \emph{logic topology}, according to which a subset in the quotient is closed if its preimage is $\omega$-internal and defined over $M$. The quotient group $G/ \GO 
		M$ is a compact Hausdorff topological group of finite exponent, and thus profinite by \cite[Theorem 4.5]{rI68}. Since $M$ and $\mathcal F$ are countable, the logic topology is second countable, so there is a fundamental system of open normal subgroups $H_n/ \GO M$ with $\GO M=\bigcap_n H_n$. Now, each quotient $H_n/\GO M$ is closed of finite index, by topological compactness, and thus the preimage $H_n$ is internal and defined over $M$, as desired.
	\end{fact}

Since the $\omega$-internal normal subgroup $\GO M$ equals a countable 
intersection of internal generic symmetric neighborhoods $Z_n$, each defined 
over $M$, the elements $a'$ of any given type $\tp(a/M)$ all lie in the 
same 
coset $a\cdot \GO M$. Indeed, we need only show that $a\inv \cdot a'$ belongs 
to $Z_n$ for every $n$ in $\N$. Now, the subset $M$ is rich,   so finitely many 
translates of each $Z_n$ by elements of $G(M)$ cover $G$. In particular, the 
element $a$ belongs to $g \cdot Z_{n+1}$ for some $g$ in $G(M)$ and therefore 
so does $a'$. Hence, we deduce that \[a\inv \cdot a' = a\inv \cdot (g\cdot g\inv)\cdot a'= (g\inv\cdot a)\inv\cdot (g\inv\cdot a')\] belongs to $(Z_{n+1})\inv\cdot 
Z_{n+1}=Z_{n+1}\cdot 
Z_{n+1}\subset Z_n$, as desired. 

We now have all the ingredients in order to state our main result. 
	\begin{theorem}\label{T:main_corners}
Consider a non-standard finite group $G$ along with an internal 
relation $S$ on $G\times G$ defined over a countable rich 
subset $M$.  If  $S$ is dense and almost surely $k$-stable for some $k\ge 1$,  then for every $g$ in $\GO M$ which is dense over $M$, the internal set 
\[\Lambda_\Box(S)_g= \left\{ (x,y)\in G\times G  : \   (x,y), (x\cdot g, 
y), (x, y\cdot g) \text{ and } (x\cdot g, y\cdot g) \text{ all lie in } S \right\} \]
has positive density.
	\end{theorem} 
	\begin{proof} 
		Fix some element $g$ in  $\GO M$ which is dense over $M$.  In order to show that the internal set $\Lambda_\Box(S)_g$ has positive density, we need only show that $\Lambda_\Box(S)_g$ contains a pair $(a, b)$ dense over $M\cup\{g\}$. We do so with the help of the following auxiliary claims.
		\begin{claim}\label{Claim:prod_inv}
			There is a tuple $(c, d)$ in good position over $M$ 
			with $c$ 
			and $d$ in $\tp(g/M)$  such that $g=c\inv\cdot 
			d$.
		\end{claim}
		\begin{claimproof}
By Fact \ref{F:G00} applied to  the 
dense element $g\inv$ of $\GO M$ and the dense type $\tp(g/M)$, 
we deduce that there is some 
element $a'$ in $\tp(g/M)$ dense over $M\cup\{g\}$ such that 
$g''=g\inv \cdot a'$ belongs to $\tp(g/M)$. Hence, for the element 
$g''$ in $\tp(g/M)$ there are two elements $c'=g$ and $d'=a'$ in 
$\tp(g/M)$ with $(c', d')$ in good position over $M$ such that 
$g''=(c')\inv \cdot d'$.  By Remark \ref{R:aleph_1} and Lemma \ref{L:Type_dense_coord}, we deduce that the 
same is true for every element in  $\tp(g''/M)=\tp(g/M)$. We thus obtain 
the desired tuple $(c, d)$ as in the statement with $g=c\inv\cdot d$. 
		\end{claimproof}
	\begin{claim}\label{Claim:translate_dense}
The internal set $S\cap S\cdot  (1_G, g\inv) $ is dense. 
	\end{claim}
		\begin{claimproof}
		Recall that $S$ is dense and thus so is $S\cdot(1_G, g\inv)$.  
		Hence, by Fact \ref{F:Measure_eq} (c), so is the 
		intersection $S\cdot(1_G, c\inv)\cap S\cdot(1_G, d\inv)$, since 
		the tuple $(c, d)$ as in Claim \ref{Claim:prod_inv} is in good 
		position over $M$ and both $c$ and $d$ lie in the dense type 
		$\tp(g/M)$.  Multiplying on the right by $(1_G, c)$, we deduce that $S\cap S 
		\cdot (1_G, 
		g\inv)$ is dense, as desired. 
				\end{claimproof}
		
		By Claim \ref{Claim:translate_dense} and Fact 
		\ref{F:UdiExercise}, there is a pair $(a_0, b)$ in good position over 
		$M\cup \{g\}$ contained in $S\cap S\cdot (1_G,  g\inv)$. That 
		is, both pairs $(a_0, b)$ and $(a_0, b\cdot g)$ belong to $S$.  
		
		 Note that we have not yet used that $S$ is almost surely stable. The 
		 idea is to replace $a_0$ by a suitable realisation $a$ such that the product
		 $ a\cdot g$  lies again in $\tp(a/M)$. 
		 
		 Indeed, Fact \ref{F:G00} (\ref{I:action}) yields that 
		 there is some $a$ in  $\tp(a_0/M)$ dense over $M\cup\{g\}$ such 
		 that  $a\cdot g$ lies again in $\tp(a_0/M)$. We may assume that 
		 $a$ is 
		 dense over $M\cup\{b,g\}$, by Lemma \ref{L:ExtensionDense}. Clearly, 
		 each of the pairs 
		 $(b, a)$, $(b, a\cdot g)$, $(b\cdot g, a)$ and $(b\cdot g, a\cdot g)$ 
		 is in good  position over $M$. 
		 
		 Theorem \ref{T:robuststable_main} implies that 
		 \[\GP(\tp(a_0/M), \tp(b/M))\subseteq  S \]
		 since the pair $(a_0, b)$ lies in $S$ and is in good position over 
		 $M$. Thus both pairs $(a, b)$ and $(a\cdot g, b)$ lie in $S$. 
		 
		 Now the pair $(a_0, b\cdot g)$ also lies in $S$ and is in good position 
		 over $M$. Again by Theorem \ref{T:robuststable_main}, we 
		 conclude that the pairs $(a, b\cdot g)$ and $(a\cdot g, b\cdot g)$ of 
		 $\GP(\tp(a'/M), \tp(b\cdot g/M))$ must 
		 lie in $S$.  In particular, the pair $(a, b)$ is dense over $M\cup\{g\}$ and belongs to the internal set $\Lambda_\Box(S)_g$, as desired.
	\end{proof}
 A straightforward application of Fubini-Tonnelli (Remark \ref{R:Fubini}) yields the following corollary.
\begin{cor}\label{C:manysquares}
Consider a non-standard finite group $G$ along with an internal 
relation $S$ on $G\times G$ defined over a countable rich 
subset $M$.   If  $S$ is dense and almost surely $k$-stable for some $k\ge 1$,  then the internal set 
\[\Lambda_\Box(S)= \left\{ (x,y, g)\in G^3   : \   (x,y), (x\cdot g, 
y), (x, y\cdot g) \text{ and } (x\cdot g, y\cdot g) \text{ all lie in } S \right\} \]
 has positive density. 
 \qed
\end{cor}


In particular, $S$ contains a dense collection of both BMZ and naive corners, as remarked
in the introduction.  Moreover, if $G$ is abelian, then there is a dense collection of triples $(x, y, d)$ in $G^3$ such that $(x, y)$, $(x+d, y)$, $(x, y+d)$, $(x+d, y+d)$ form a square in $S$.

In the special case of a binary relation given by the Cayley graph of 
a subset $A$ of $G$, the proof of Theorem \ref{T:main_corners} 
yields a stronger result, since every dense element $g$ in $\GO M$ is 
a popular side length for many corners. 
\begin{prop}\label{P:AP_robustlystable}
Consider a non-standard finite group $G$ and an internal 
almost surely $k$-stable dense subset $A$ of $G$ defined over a countable 
rich subset $M$.  For every $g$ in $\GO M$ which is dense over $M$, the sets 
$A$ and $g\cdot A$ are comparable, that is,  \[ \mu_G(A\triangle 
(g\cdot A))=0.\] 

In particular, whenever the element $a$ in $A$ is dense over $M\cup\{g\}$, 
$g^m\cdot a$ belongs to $A$ for all $m$ in $\Z$. Hence, the 
 $\omega$-internal set $\bigcap_{m\in \Z} g^m\cdot A$ is dense with measure 
 $\mu_G(A)>0$. 
\end{prop}
\begin{proof}
As in Claim \ref{Claim:prod_inv} of Theorem \ref{T:main_corners}, 
write $g=c\inv \cdot d$, where $(c, d)$ are in good position over $M$ 
and both lie in the same type $\tp(g/M)$. Suppose towards a contradiction 
that $\mu_G(A\triangle (g\cdot A))>0$, or  equivalently, that 
\[ \mu_G\big((c\cdot A) \triangle  (d\cdot A))>0.\]
Without loss of generality, we may assume that $(c\cdot A)\setminus 
(d\cdot A)$ is dense, so it contains some element $b$ which is dense over 
$M\cup\{c,d\}$. This means that $c\inv\cdot b$ belongs to $A$ but 
$d\inv \cdot b$ does not, or equivalently, the pair $(c, b)$ belongs to the almost surely
$k$-stable relation $\mathrm{Cay}(G,A)$, but $(d, b)$ does not. Since both 
pairs belong 
to $\GP(\tp(g/M), \tp(b/M) )$, this contradicts Theorem 
\ref{T:robuststable_main}.

To prove the final part, observe that since $A$ and $g\cdot A$ are comparable, so are $A$ and 
$g\inv \cdot A$. By continuity, we need only show that each 
intersection \[ 
\bigcap_{-m\le i\le m} g^i\cdot A \]
has constant measure $\mu_G(A)$ for all $m$ in $\Z$. Otherwise, the internal 
set \[ A \setminus \bigcap_{-m\le i\le m} g^i\cdot A \] 
is dense, so by Lemma 
\ref{L:ExtensionDense} we may choose an element $a$ in 
$A$ which is dense over $M\cup\{g\}$. Remark \ref{R:Comp} implies that 
$a$ must belong to both $g\cdot A$ and  $g\inv \cdot A$, and inductively, we 
conclude that $a$ 
belongs to $g^{i} \cdot A$ for all $-m\le i\le m$, which gives the desired 
contradiction. 
\end{proof}

We conclude this section by deducing a finitary (albeit 
ineffective) version of
Corollary \ref{C:manysquares}, using Lemma \ref{L:Transfer} and the accompanying Remark \ref{R:Transfer}, and strengthening it in two special cases.

\begin{cor}\label{C:corners}
Given an integer $k\ge 1$ and a real number $\delta>0$, there is an integer 
$\ell_0(k, \delta)\ge 1$ and real numbers $\theta=\theta(k, \delta)>0$ and  
$\epsilon=\epsilon(k, \delta)>0$ with the following property. 

Let $G$ be a finite group of order $|G|\ge \ell_0$ and let  $S\subseteq 
G\times G$  be a relation of size $|S|\ge \delta |G|^{2}$ such that the collection  
$\HH_k(S)$ of all 
half-graphs of height $k$ induced by $S$ on  $G$ has size $|\HH_k(S)|\le \theta 
|G|^{2k}$. Then $\Lambda_\Box(S)$ 
has size at least $\epsilon |G|^{3}$. 

In particular, the relation $S$ contains a (non-trivial) square provided that $|G|\geq 1/\epsilon$.
\end{cor}
\begin{proof}
The proof proceeds by contradiction. Suppose there are $k\ge 1$ and 
$\delta>0$ such that for every $\ell\ge 1$ (setting $\theta=\epsilon=1/\ell$), we can find a finite group $G_\ell$ of order at least 
$\ell$ 
and a relation $S_\ell\subseteq G_\ell\times G_\ell$ of density $\delta$ such that 
$\HH_k(S_\ell)\le \frac{1}{\ell} |G_\ell|^{2k}$, yet 
$|\Lambda_\Box(S_\ell)|<|G_\ell^{3}|/\ell$.

In particular, the family $(G_\ell,  S_\ell)_{\ell\in \N}$ is 
almost surely $k$-stable, as in Remark \ref{R:robust_asympt}. Take a non-principal 
ultrafilter $\Ult$ on $\N$ and consider the non-standard finite group 
$G=\prod_{\ell\to \Ult} G_\ell$ as well as the internal set $
S=\prod_{\ell\to \Ult} 
S_\ell$. By Remark \ref{R:Transfer}, the set $S$ has Loeb measure 
$\mu_{G^{2}}(S)\ge \delta$ and is almost surely $k$-stable by  Remark 
\ref{R:robust_asympt}. Choose any countable rich subset $M$. Corollary \ref{C:manysquares} yields that $\mu_{G^{3}}(\Lambda_\Box(S))$ is at least $\eta$ for some $\eta>0$.

Remark \ref{L:Transfer} applied to $r=\eta$ yields 
the desired contradiction, since  
$\Lambda_\Box(S_\ell)$ has size at most $\eta|G_\ell|^{3}/2$ for sufficiently large $\ell$.   
\end{proof}

In the particular case that the group in question has bounded exponent, we 
can strengthen Corollary \ref{C:corners} using Fact \ref{F:G00} to obtain a 
subgroup of bounded index almost all of whose elements witness the 
existence of a square. 
\begin{cor}\label{C:corners_bddexp}
	Given integers $k, r\ge 1$ and real numbers $\delta, \epsilon>0$, there is 
	an integer 
	$\ell=\ell(k, r, \delta, \epsilon)\ge 1$ and real numbers
	$\theta=\theta(k, r,  
	\delta, \epsilon)>0$ and $\eta=\eta(k, r, \delta, \epsilon)>0$ with the following property.

Let  $G$ be a finite group of exponent bounded by $r$, and consider a relation $S\subseteq 
	G\times G$  of size $|S|\ge \delta |G|^{2}$ such that the collection  
	$\HH_k(S)$ of all 
	half-graphs of height $k$ induced by $S$ on  $G$ has size $|\HH_k(S)|\le 
	\theta 
	|G|^{2k}$. Then there exists a subgroup $H$ of $G$ of index at most 
	$\ell$ such that 
	\[  \left|\left\{ g\in H  : \ |\Lambda_\Box(S)_g| <\eta |S|\right\} \right| <  \epsilon |H|,\]  where $\Lambda_\Box(S)_g\subseteq S$ is as in Theorem \ref{T:main_corners}.
\end{cor}
\begin{proof}
As in the proof of Corollary \ref{C:corners}, negating quantifiers, we deduce 
that for a fixed choice of $k, r, \delta$ and $\epsilon$, for every 
$\ell$ in $\N$ 
(setting $\theta=\eta=1/\ell$) there is a finite group $G_\ell$ of exponent at most $r$ and a 
relation 
$S_\ell\subseteq G_\ell\times G_\ell$ of density $\delta$ such 
that 
\[ |\HH_k(S_\ell)|\le \frac{1}{\ell} |G_\ell|^{2k},\] yet for 
every 
subgroup $H\leqslant G_\ell$ of index 
at most $\ell$, the subset of elements $g$ in 
$H$ with \[ |\Lambda_\Box(S)_g| <\frac{|S|}{\ell} \]  has size at least 
$\epsilon |H|$. In particular, the group $G_\ell$ has size at least $\ell$, by taking $H$ the trivial subgroup of $G_\ell$.

Take a non-principal 
ultrafilter $\Ult$ on $\N$ and consider the non-standard finite group 
$G=~\prod_{\ell\to \Ult} G_\ell$ as well as the internal set $ S=\prod_{\ell\to \Ult} 
S_\ell$. By Remark \ref{R:Transfer}, the set $S$ has Loeb measure 
$\mu_{G^{2}}(S)\ge \delta$ and is almost surely $k$-stable by  Remark 
\ref{R:robust_asympt}.  Choose any countable 
rich subset $M$. Notice that the exponent of the non-standard finite group $G$ is 
bounded by $r$ by \L o\'s's Theorem (Theorem \ref{R:internal}). Fact \ref{F:G00} yields 
that the subgroup $\GO M$ is a countable intersection of internal subgroups $H_n$, each of finite index and defined over $M$.  We may assume that $H_{n+1}\subseteq H_n$ for every $n$ in $\N$.

Definition \ref{D:Tame} tells us that the set \[
    Z=\{g \in \GO  M : \ \mu_{G^{2}}(\Lambda_\Box(S)_g)=0\} = \bigcap\limits_{\substack{n\in \N \\ 0\ne m\in \N}} \left\{g \in H_n : \ \mu_{G^{2}}(\Lambda_\Box(S)_g)<\frac{1}{m}\right\}  \] is $\omega$-internal.

Fact \ref{F:UdiExercise} and Theorem \ref{T:main_corners} imply that $Z$ is not dense, so by $\aleph_1$-saturation and Remark \ref{R:Dense}, we 
deduce that for some $0\ne m$ in $\N$  and some finite-index internal subgroup $H=H_n$ 
defined over $M$, the set $\{g \in H  : \   \mu_{G^{2}}(\Lambda_\Box(S))_g) < \frac{1}{m}\}$ is contained in some internal set of density $0$. 

Choose some $\ell_0$ in $\N$ with \[ 0<\frac{\mu_{G^{2}}(S)}{\ell_0}< \frac{1}{2m}\] such that for $\Ult$-almost all $\ell\ge \ell_0$,  the subgroup $H(G_{\ell})$ has index $[G:H]\le \ell_0$.  Now, Definition \ref{D:Tame} yields an internal set $W$ with  \[\left\{g\in H  : \  \mu_{G^{2}}(\Lambda_\Box(S))_g) < \frac{1}{2m} \right\}\subseteq W\subseteq \left\{g\in H  : \  \mu_{G^{2}}(\Lambda_\Box(S))_g) \le \frac{1}{2m}\right\},\] whence $\mu_G(W)=0$ and $ \mu_G(W)<\epsilon\cdot \mu_G(H)=\epsilon/[G:H]$. Therefore, for $\Ult$-almost all $\ell\ge \ell_0$,  we have that 
 \[ \left|\left\{g \in H(G_\ell) : \ |\Lambda_\Box(S_\ell)_g| < \frac{|S_\ell|}{\ell} \right\} \right| \le |W(G_\ell)|<
\frac{\epsilon}{[G:H]} |G_\ell| = \epsilon |H(G_\ell)|.\] This 
contradicts our choice 
of $G_\ell$, as desired. 
\end{proof}

Mimicking the proof above, together with Proposition \ref{P:AP_robustlystable}, 
we deduce the 
following finitary statement concerning the existence of arbitrary long arithmetic 
progressions in almost surely stable dense sets.

\begin{cor}\label{C:AP_robust_bddexp}
	Given integers $k, m, r\ge 1$ and real numbers $\delta, \epsilon, 
	\eta>0$, there 
	is
	an integer 
	$n=n(k, m, r, \delta, \epsilon, \eta)\ge 1$ and a real number 
	$\theta=\theta(k, m, 
	r,  
	\delta, \epsilon, \eta)>0$ with the following property.
	
	Let $G$ be a finite group of exponent bounded by $r$, and let $A\subseteq G$ be subset of $G$ of size $|A|\geq \delta|G|$ such that the collection  
	$\HH_k(\mathrm{Cay}(G,A))$ of all 
	half-graphs of height $k$ induced by its Cayley graph on  $G$ has size 
	$|\HH_k(\mathrm{Cay}(G,A))|\le 
	\theta 
	|G|^{2k}$. Then there exists a subgroup $H$ of $G$ of index at most 
	$n$ such that 
	\[  \left| \{ h \in H  : \ |\text{m-}{\rm AP}(A)_h| 
	<(1-\eta)|A|\} \right| < \epsilon |H|, \]
	where $\text{m-}{\rm AP}(A)_h= \{ a\in A  : \  \{a,h\cdot 
	a,\ldots,h^{m-1}\cdot a\} \subseteq A\}$. 
\end{cor}

\section{Grids and $L$-shapes}\label{S:grid}

In the preceding section we saw that almost sure stability was sufficient to imply the existence of squares in dense subsets of Cartesian products of arbitrary finite groups. In order to extend the previous results to other $2$-dimensional shapes, such as $L$-shapes or $3\times 2$-grids, we will need to impose some (mild) conditions on the nature of the groups in question. 

As mentioned in the introduction, in recent work \cite{sP22} Peluse obtained the first reasonable upper bound on the density of $2$-dimensional subsets of $\mathbb F_p^n$ without $L$-shaped configurations. We will show how to obtain a qualitative version of her result valid in 
finite abelian groups of odd order under the assumption of almost sure stability.

 Given a subset $X$ of $G$, every solution $(x, y, z)$ in $X^3$ to the equation 
 \[ 
 x \cdot y= z^2
 \] 
 with $x\ne y$ determines a \emph{generalised arithmetic progression of length} $3$ in $X$. Indeed, if $x\cdot y=z^2$ with $x\ne y$, then $g= x\inv\cdot z= y\cdot z\inv$ is non-trivial and satisfies that all three elements $x, x\cdot g$ and $g\cdot x\cdot g$ belong to $X$. If the group $G$ is abelian, this is an arithmetic progression in the classical sense, for $g\cdot x\cdot g=g^2\cdot x$.   In what follows, we denote by $\EE(X)$ the set \[ \EE(X)=\{ (x,y) \in X^2 :\  x\cdot y= z^2 \text{ for some $z$ in } X \}.\]

\begin{definition}\label{D:suitable}
 The non-standard finite group $G$ is \emph{suitable} for the equation $x\cdot y=z^2$ if the following conditions hold:
\begin{itemize}
\item[(i)] \emph{Squaring preserves dense elements} (cf. \cite[Definition 3.12]{MPP21}), that is, for every element $a$ and 
every countable set of parameters $B$, we have that $a$ is dense 
over $B$ if and only if the element $a^2$ is dense over $B$. 
\item[(ii)] For every dense internal subset $X$ of $G$, the internal set $\EE(X)$ 
has positive density in $G^2$. 
\end{itemize}
 \end{definition}
 Whilst (i) implies (ii) in any non-standard finite group $G$ (see \cite[Theorem 3.14]{MPP21}), for the sake of a self-contained presentation we have decided to impose (ii) as an additional condition, since the examples we are most interested in already verify this condition by classical results in additive combinatorics, as the next remark shows. 
\begin{remark}\label{R:ultr_suitable}
Every non-standard finite group $G$ obtained as a non-principal ultraproduct of a family $(G_\ell)_{\ell\in \N}$ of finite abelian groups of odd order is suitable for the equation $x\cdot y=z^2$.  Indeed, in such a  non-standard finite abelian group $G$ as above there are no involutions, so $a$ is the only element in the fiber of $a^2$ with respect to the group homomorphism $x\mapsto x^2$. Thus, density of elements is preserved, by Fubini-Tonelli (Remark \ref{R:Fubini}). 

There is a plethora of explicit lower bounds for $\EE(X)$ for finite abelian groups of odd order. For instance, Bloom and Sisask showed in \cite[Theorem 2.1]{BS19} that whenever a finite subset $X_\ell$ of a finite abelian group $G_\ell$ of odd order has density $|X_\ell|/|G_\ell|\ge \sigma$, then $|\EE(X_\ell)|\ge f(\sigma) |G_\ell|^2$, with \[ f(\sigma)= \sigma^2\exp(-C \sigma\inv\log(\sigma^{-1})^C)\] for some absolute constant $C>0$.\footnote{For the best bound available at the time of writing, see \cite{BS23}, which is based on the breakthrough work of Kelley and Meka \cite{KM23}.} It follows easily from  \L o\'s's Theorem that for every  internal subset $X$ of density $\mu_G(X)\ge \sigma$ in the non-standard finite abelian group $G$, $\mu_{G^2}(\EE(X))\ge f(\sigma)>0$, as desired. 
\end{remark}

\begin{theorem}\label{T:Lshape}
Consider a non-standard finite group $G$ suitable for the equation $x\cdot y=z^2$ along with an internal 
relation $S$ on $G\times G$ defined over a countable rich 
subset $M$ of $G$. If $S$ is dense and  almost surely $k$-stable for some $k\ge 1$,  then there is a tuple $(a, g, b)$ of $G^3$ in good position over $M$ such that $(a, b, g)$ belongs to the internal set 
\[ \Lambda_{3\times 2}(S)= \left\{ (x, y,h) \in G^3 : 
\  
\parbox{8cm}{$(x, y), (x\cdot h, y), (h\cdot x\cdot h, y),  (x, y\cdot h), \\ (x\cdot h, y\cdot h) \text{ 
	and }  (h\cdot x\cdot h, y\cdot h) \mbox{ all lie in } S$}
\right\}.\] 
In particular, the set $\Lambda_{3\times 2}(S)$ has positive density in $G^3$, by Fact \ref{F:UdiExercise} and Definition \ref{D:Tame} (II). 
\end{theorem}
\begin{proof}
Set $\delta=\mu_{G^{2}}(S)/2>0$. By Definition \ref{D:Tame}, 
there is an internal set $Y_\delta$ defined over $M$ such that \[ 
\{y\in G  : \  
\mu_G(S_y)> \delta \}\subseteq Y_\delta\subseteq \{y\in G : 
\ 
\mu_G(S_y)\ge \delta \},\] where $S_y=\{x\in G  : \  (x, y) \in S\}$ denotes the fiber of $S$ over $y$. By Fubini-Tonelli 
(Remark \ref{R:Fubini}), the internal set $Y_\delta$
 is dense. Hence, choose some $c$ in $G$ dense over $M$ with $\mu_G(S_{c})\ge \delta$. By Remark \ref{R:LS}, there is a countable rich 
subset $N$ containing $M\cup\{c\}$. Notice that the internal dense subset 
$S_{c}\subseteq G$ is defined over $N$.

\begin{claim*}
There exists a tuple $(a_1, a_2, a_3)$ in $S_c^3$ such that
\begin{itemize}
    \item[(i)] $a_1\cdot a_2=a_3^2$;
    \item[(ii)] the difference $g=a_1\inv \cdot a_3=a_2\cdot a_3\inv$ is dense over $N$ and belongs to $\GO N$; 
    \item[(iii)] the pair $(a_1, g)$ is in good position over $N$. 
\end{itemize}
\end{claim*}

\begin{claimproof*}
It suffices to find a triple satisfying (i) and (ii) in the statement of the claim with $(a_1,a_3)$ in good position over $N$, or equivalently, with $(a_1,a_2)$ in good position over $N$, since squaring preserves dense elements.

Write $\GO N$ as a countable intersection of internal generic symmetric neighbourhoods $Z_j$, with $j$ in $\N$, defined over $N$ as in Fact \ref{F:G00} (\ref{I:nbhood}). In particular, for each $Z_j$ there are finitely many elements $t_1(j),\ldots, t_{n_j}(j)$ in $G$ such that \[ G=Z_j\cdot t_1(j)\cup\cdots\cup Z_j\cdot t_{n_j}(j).\]
Note that we may find such elements $t_r(j)$ in $N$, since $N$ is a rich subset of $G$. Moreover, the product set  $Z_{j+1}\cdot Z_{j+1}$ is a subset of $Z_j$. Assume towards a contradiction that there is no triple as required, and consider the $\omega$-internal set  
 \[
 \mathcal Z=  \bigcap\limits_{\substack{j\in \N \\ \mathcal X \in \mathcal F}} \left\{ (x,y,z) \in S_c^{3}  : \ \text{$x\cdot y= z^2$ with  $x\inv \cdot z=y\cdot z\inv$ in $Z_j$  and  $(x,y) \notin \mathcal X$}  \right\},
 \] 
 where $\mathcal F$ is the countable collection of all internal subsets $\mathcal X$ of $G^2$ defined over $N$ of density $0$.  Notice that the projection $\pi$ of $\mathcal Z$ onto the first two coordinates is again $\omega$-internal by Remark \ref{R:InfInt} and empty, for otherwise it would contain a pair $(a_1,a_2)$ in good position over $N$, by Fact \ref{F:UdiExercise}, but such a pair would yield a triple as in the statement.  Therefore, the $\omega$-internal set $\mathcal Z$ must be empty as well, since $\pi(\mathcal Z)=\emptyset$. By Fact \ref{F:Saturation}, there are finitely many sets $\mathcal X_1, \ldots, \mathcal X_r$ in $\mathcal F$ and some $j$ in $\N$ such that the internal set 
   \[ \widetilde{\mathcal Z}=\left\{ (x,y,z) \in S_c^{3}  : \    x\cdot y=  z^2  \text{ with $x\inv \cdot z=y\cdot z\inv$ in $Z_j$ } \right\}\] is covered by $\bigcup_{i=1}^r \pi\inv(\mathcal X_i)$. Now, the internal set $S_c$ is dense, so there exists some $t_r(j+1)$ in $N$ such that $S_c\cap (Z_{j+1}\cdot t_r(j+1))$ is dense as well. It follows from the suitability of $G$ that the corresponding set $\EE(S_c\cap (Z_{j+1}\cdot t_r(j+1))$ must have positive density. Hence, there is a dense pair $(u_1,u_2)$ over $N$ with each $u_i$ in $S_c\cap(Z_{j+1}\cdot t_r(j+1))$ such that \[ u_1\cdot u_2=  u_3^2 \text{  for some $u_3$ in  $S_c\cap(Z_{j+1}\cdot t_r(j+1))$.} \] 
Now the pair $(u_1,u_2)$ avoids all $\mathcal X_i$ and the common difference 
\[
u_1\inv\cdot u_3 = u_2\cdot u_3\inv=u_2\cdot (t_r(j+1)\inv \cdot  t_r(j+1) ) \cdot u_3\inv =( u_2\cdot t_r(j+1)\inv ) \cdot (u_3\cdot t_r(j+1)\inv)\inv 
\]
belongs to $Z_{j+1}\cdot Z_{j+1}\subseteq Z_j$. Thus, the triple $(u_1,u_2, u_3)$ belongs to $\widetilde{\mathcal Z}$ and thus to some $\pi\inv(\mathcal X_j)$. This implies that  the dense pair $(u_1,u_2)$ over $N$ belongs to the internal set $\mathcal X_j$ defined over $N$ of density $0$, which yields the desired contradiction. 
\end{claimproof*}

Choose now a tuple $(a_1,a_2,a_3)$ in $S_c^{3}$ as in the Claim with  $a_1\cdot a_2 = a_3^2$ , where $g=a_1\inv \cdot a_3=a_2\cdot a_3\inv$ lies in $\GO N$ and $(a_1,g)$ in good position over $N$.   By Fact \ref{F:G00} (\ref{I:GNsbgpGM}),  we have that $g$, viewed as an element of $\GO M$, is dense over $M$.  Fact \ref{F:G00} (\ref{I:action}) yields that for some $b$ in $\tp(c/M)$ dense over $M\cup\{a_1, a_2,a_3\}$, the product $b\cdot g$ also belongs to $\tp(c/M)$ (and is dense over $M\cup\{a_1, a_2, a_3\}$). Hence, the triple $(a_1,g, b)$ is in good position over $M$. We need only verify that $(a_1, b, g)$ belongs to $\Lambda_{3\times 2}(S)$. 

By construction, the points 
\[ 
(a_1, c), (a_1\cdot g,c) \text{ and } (g\cdot a_1\cdot g, c) \] all lie in the almost surely $k$-stable internal relation $S$. As each pair (with the order of coordinates reversed) is in good position over $M$, Theorem \ref{T:robuststable_main} implies that for each $i=1,2,3$,
\[\GP(\tp(a_i/M), \tp(c/M))  \subseteq  S,\] so the pairs \[ (a_1, b), (a_1\cdot g,b) \text{ and } (g\cdot a_1\cdot g, b) \] all lie in the almost surely 
$k$-stable internal relation $S$. Analogously, using now that the element $b\cdot g$ of $\tp(b/M)$ is also dense over $M\cup\{a_1, a_1\cdot g, g\cdot a_1\cdot g\}$, we conclude again by Theorem \ref{T:robuststable_main} that each of the pairs \[ (a_1, b\cdot g), (a_1\cdot g, b\cdot g) \text{ and } (g\cdot a_1\cdot g, b\cdot g)\] lies in the almost surely
$k$-stable internal relation $S$, as desired.
\end{proof}
Remark \ref{R:ultr_suitable} immediately yields the following result. 
\begin{cor}\label{C:Lshapes_4AP}
	If the non-standard finite group $G$ is obtained as an ultraproduct of finite abelian 
	groups of odd order, then for every almost surely $k$-stable dense internal relation $S$ of $G\times G$ the collection of $3\times 2$ grids \[ \Lambda_{3\times 2}(S)= \left\{ (a, b, g) \in G^3 : 
\  
\parbox{8cm}{$(a, b), (a +  g, b), (a+2g, b),  (a, b+ g), \\ (a+g, b+g) \text{ 
	and }  (a + 2g,  b + g)  \text{ all lie in } S$}
\right\}\] has positive density in $G^3$.  

In particular,  every dense almost surely $k$-stable subset $Z$ of $G$ (see Definition \ref{D:stab}) contains a $4$-term arithmetic progression $\{z, z+d, 
	z+2d, z+3d\}$ given by a dense pair $(z, d)$ (with $d\ne 0$).   
\end{cor}

Just like Theorem \ref{T:main_corners} gives rise to the finitary Corollary \ref{C:corners_bddexp}, Theorem \ref{T:Lshape} yields a finitary version for sufficiently large abelian groups of odd order. 
\begin{cor}\label{C:Lshape}
Given an integer $k\ge 1$ and a real number $\delta>0$, there is an integer 
$\ell_0(k, \delta)\ge 1$ and real numbers $\theta=\theta(k, \delta)>0$ and  
$\epsilon=\epsilon(k, \delta)>0$ with the following property. 

Let $G$ be a finite abelian group of odd order with $|G|\ge \ell_0$, and consider a relation $S\subseteq G\times G$ of size $|S|\ge \delta |G|^{2}$ such that the collection  
$\HH_k(S)$ of all 
half-graphs of height $k$ induced by $S$ on  $G$ has size $|\HH_k(S)|\le \theta 
|G|^{2k}$. Then the set 
\[ \Lambda_{3\times 2}(S)= \left\{ (a, b, g) \in G^3 :
\  
\parbox{8cm}{$(a, b), 
(a+ g, 
b), (a+2g, b),  (a, b+g), \\ (a+g,  b+ g) \text{ 
	and }  (a+ 2g, b+ g)  \text{ all lie in } S$}
\right\}\]
has size  $|\Lambda_{3\times 2}(S)| \ge \epsilon |G|^{3}$. In 
particular, the relation $S$ contains an $L$-shape. 
\end{cor}

\begin{remark}\label{R:stable_anyconfig}
The astute reader will have noticed that we used very little about the particular equation $x\cdot y=z^2$ in Theorem \ref{T:Lshape}. Indeed, if the non-standard finite group $G$ is built as an ultraproduct of finite groups $(G_\ell)_{\ell\in\N}$ such that for every finite subset $X_\ell$ of $G_\ell$ of density $\sigma$, the collection of tuples $(a_1,\ldots, a_{m+1})$ in $X_\ell^{m+1}$ satisfying a certain pattern has size at least  $f(\sigma)|G_\ell^m|$ for some function $f:\mathbb R\to \mathbb R$ which is uniform in the family $(G_\ell)_{\ell\in\N}$, then we could reproduce the proof of Theorem \ref{T:Lshape} verbatim to obtain grids using differences $g=a_i-a_j$ for any suitable choice of coordinates. 

This applies in particular to the pattern given by the equation $n_1x_1+\cdots+n_mx_m=k z$ with $k=\sum_{j=1}^m n_j$ in $\mathbb Z/p\mathbb Z$, asymptotically as $p$ is large, see for instance \cite[Theorem 3]{tK23}. 
\end{remark}


\begin{thebibliography}{99}
		
		\bibitem{ASz74} M. Ajtai and E. Szemer\'edi, \emph{Sets of 
		lattice points that form no squares}. Stud. Sci. Math. Hungar. {\bf 
		9}, 9--11 (1974).
		
	\bibitem{A16} T. Austin, \emph{Ajtai–Szemer\'edi theorems over quasirandom groups}. In \emph{Recent Trends in Combinatorics}, Vol. 159 of The IMA Volumes in Mathematics and its Applications, Springer, pp. 453--484 (2016).
	
		\bibitem{BB21} M. Bays and E. Breuillard, \emph{Projective 
		geometries arising from Elekes-Szab\'o problems}, Ann. Sci. \'Ec. 
		Norm. Sup\'er. {\bf 54},  627--681  (2021).

		
		\bibitem{BMZ97} V. Bergelson, R. McCutcheon and Q. Zhang. 
		\emph{A Roth theorem for amenable groups}. Amer. J. Math. 
		{\bf 119}, 1173--1211 (1997).
		
		\bibitem{BT14} V. Bergelson and T. Tao, \emph{Multiple 
		recurrence in quasirandom groups}, Geom. Funct. Anal. {\bf 24}, 1--48 (2014).
		
		\bibitem{BRZK17} V. Bergelson, D. Robertson and P. Zorin-Kranich, \emph{Triangles in Cartesian Squares of Quasirandom Groups}, Combinatorics, Probability and Computing {\bf 26}, 161--182 (2017).

		\bibitem{BS19} T. F. Bloom and O. Sisask, \emph{Logarithmic bounds for Roth’s theorem via almost-periodicity}, Discrete Anal., Paper No. 4, 20 pp.  (2019).
		
		\bibitem{BS23} T. F. Bloom and O. Sisask, \emph{An improvement to the Kelley-Meka bounds on three-term arithmetic progressions}, preprint (2023), \url{https://arxiv.org/abs/2309.02353}.
		
		\bibitem{aB25} A. Burka, \emph{Notions of negligibility and approximate classification}, preprint, (2025), \url{https://arxiv.org/abs/2509.06238}.
		
			\bibitem{CT24} A. Chernikov and H. Towsner, \emph{Perfect stable regularity lemma and slice-wise stable hypergraphs}, preprint (2024), 
		\url{http://arxiv.org/abs/2402.07870}.
		
		\bibitem{CPT20} G. Conant, A. Pillay and C. Terry. \emph{A 
			group version of stable regularity}, Math. Proc. Camb. Philos. Soc. 
		{\bf 168} , 405--413 (2020).
		
		
		\bibitem{gC21} G. Conant, \emph{Quantitative structure of 
			stable sets in arbitrary finite groups}, Proc. Amer. Math. Soc. {\bf 
			149},  4015--4028 (2021).
  
		\bibitem{CM22} L. Coregliano and M. Malliaris, \emph{Countable Ramsey}, J. Math. Logic, {\bf 2550021},  65 pp. (2025), DOI: 10.1142/S0219061325500217.
		
			
		\bibitem{lvD15} L. van den Dries, \emph{Approximate groups}, 
		Ast\'erisque {\bf 367--368}, 79--113 (2015). 
		
		
		\bibitem{EFR86} P. Erd\H{o}s, P. Frankl, and V. R\"odl, 
		\emph{The asymptotic number of graphs not
			containing a fixed subgraph and a problem for hypergraphs 
			having no exponent},
		Graphs Combin. {\bf 2} , 113–121 (1986).
		
		\bibitem{FK78} H. Furstenberg and Y. Katznelson, \emph{An 
		ergodic Szemer\'edi theorem for commuting transformations}, J. 
		Analyse Math. {\bf 34}, 275--291 (1978).
		
		\bibitem{mG25} M. Gir\'on, \emph{On the regularity of almost stable relations}, preprint (2025), \url{https://arxiv.org/abs/2508.00511}. 
	
	
		\bibitem{G07} W. T. Gowers, \emph{Hypergraph regularity and 
		the multidimensional Szemer\'edi theorem}, Annals of 
		Mathematics {\bf 166}, 897--946 (2007).
		
	
		\bibitem{G05} B. Green, \emph{Finite field models in additive 
		combinatorics}, In \emph{Surveys in Combinatorics 2005}, 
		London Math. Soc. Lecture Note Ser., vol. 327, Cambridge Univ. 
		Press, Cambridge, 1--27 (2005).
		
	
		\bibitem{eH12} E. Hrushovski, \emph{Stable group theory and 
		approximate
			subgroups}, J. Amer. Math. Soc. {\bf 25}, 189--243 (2012).
		
		\bibitem{HP94} E. Hrushovski and A. Pillay, \emph{Groups definable in 
		local
			fields
			and pseudo-finite fields}. Israel J. Math. {\bf 85},  
			203--262 (1994).
		
	
		
		\bibitem{rI68} R. Iltis, {\em Some algebraic structure in the dual 
		of a compact group}, Canad. J. Math. {\bf 20}, 
		1499--1510 (1968).
	
		\bibitem{J25} M. Jaber, Y.P. Liu, S. Lovett, A. Ostuni, M. Sawhney, 
	\emph{Quasipolynomial bounds for the corners theorem},
	preprint (2025), 
	\url{https://arxiv.org/abs/2504.07006}.
	
	
	\bibitem{KM23} Z. Kelley and R. Meka, \emph{Strong bounds for 3-progressions}, 2023 IEEE 64th Annual Symposium on Foundations of Computer Science (FOCS), Santa Cruz, CA, USA, pp. 933--973 (2023).    
	
		\bibitem{KP97} B. Kim and A. Pillay, \emph{Simple theories}, 
		Annals Pure Applied Logic {\bf 88}, 149--164 (1997).
   

 \bibitem{tK23} T. Ko\'sciuszko. \emph{Counting solutions to invariant equations in dense sets}, preprint (2023),  \url{https://arxiv.org/abs/2306.08567}. 
        
		\bibitem{LM07} M. Lacey and W. McClain, \emph{On an 
		argument of Shkredov on two-dimensional corners}. Online J. 
		Anal. Comb. {\bf  2},  21 pp. (2007). 

  \bibitem{MOS18} S. Montenegro, A. Onshuus and P. Simon, \emph{ Groups with
f-generics in NTP$2$ and PRC fields}, J. Inst. Math. 
Jussieu {\bf 19}, 821--853 (2019).
		
	
	
		 \bibitem{MPPW21} A. Martin-Pizarro, D. Palac\'in and J. Wolf, \emph{A 
		 model-theoretic note on the Freiman-Ruzsa theorem}, Sel. Math. New 
		 Ser. {\bf 27}, Paper No. 53, 19 pp. (2021).
		
		\bibitem{MPP20} A. Martin-Pizarro and D. Palac\'in, 
		\emph{Stabilizers, Measures and IP-sets}, Notre Dame J. Form. Log. {\bf 66}, 189--204  (2025).
		
		\bibitem{MPP21} A. Martin-Pizarro and D. Palac\'in,
		\emph{Complete Type amalgamation for non-standard finite groups}, Model Theory {\bf 3}, 1--37 (2024).
		
		\bibitem{NRS06} B. Nagle, V. R\"odl and M. Schacht, \emph{The 
		counting lemma for regular k-uniform hypergraphs}, Random 
		Structures and Algorithms {\bf 28}, 113--179 (2006).
		
		\bibitem{dP20} D. Palac\'in, \emph{On compactifications and 
		product-free sets}, J.  London Math. Soc. {\bf 101}, 
		156--174 (2020).


   	\bibitem{sP22} S. Peluse, \emph{Subsets of 
		$\mathbb{F}_p^n\times\mathbb{F}_p^n$ without L-shaped 
		configurations}, Compos. Math. {\bf 160}, 176--236 (2023).
	
		\bibitem{kR53} K. Roth, \emph{On certain sets of integers}, 
		J. London Math. Soc. {\bf 28},  104--109 (1953).
		
		
		\bibitem{RS04} V. R\"odl and J. Skokan, \emph{Regularity 
		lemma for k-uniform hypergraphs}, Random Structures and 
		Algorithms {\bf  25},   1--42 (2004).
		
			 \bibitem{sSbook} S. Shelah, \emph{Classification theory and 
			 the 
			number of
			nonisomorphic models.}, Studies in Logic and the Foundations 
		of Mathematics
		{\bf 92}, North-Holland Publishing Co., Amsterdam-New York, ISBN
		0-7204-0757-5, (1978).
	
		\bibitem{S05} I. Shkredov, \emph{On a problem of Gowers}. 
		Dokl. Akad. Nauk {\bf 400}, 169--172 (2005). 
		
		\bibitem{S06} I. Shkredov, \emph{On a problem of Gowers}. Izv. 
		Ross. Akad. Nauk Ser. Mat. {\bf 70}, 179--221
		(2006). 
		
		
		\bibitem{S13} J. Solymosi, \emph{Roth-type theorems in finite groups}. European J. Combin. {\bf 34}, 1454--1458 (2013).
		\bibitem{TW19} C. Terry and J. Wolf, \emph{Stable arithmetic 
		regularity lemma in the finite-field model}. Bull. London Math. Soc. {\bf 51}, 70--88 (2019).
		
		\bibitem{TW21} C. Terry and J. Wolf, \emph{Higher-order generalizations 
		of stability and arithmetic regularity}, preprint (2021), 
		\url{http://arxiv.org/abs/2111.01739}.
		
		\bibitem{tS20} T. Sanders, \emph{The coset and stability rings}, 
		Online J. Anal. Comb. {\bf 15}, 10pp. (2020).
		
	\end{thebibliography}
\end{document}